\documentclass{amsart}
\usepackage{amssymb,amscd,verbatim, amsthm,amsmath,amsgen,xspace}
\usepackage{latexsym}
\usepackage{amsfonts}
\usepackage{color}
\usepackage{ulem}
\usepackage[all]{xy}

\usepackage{graphicx}

\newcommand \ch[1]{{\check{#1}}}
\newcommand \bb[1]{{\mathbb #1}}

\newcommand \bC{{\bb C}}

\newcommand \bN{{\bb N}}
\newcommand \bR{{\bb R}}
\newcommand \bZ{{\bb Z}}

\newcommand\Hom{\operatorname{Hom}}

\newcommand\rk{\operatorname{rk}}
\newcommand\Ann{\operatorname{Ann}}
\newcommand\half{\frac{1}{2}}

\newcommand{\Kt}{\widetilde{K}}
\newcommand{\Et}{\widetilde{E}}

\DeclareMathOperator\Ass{Ass}
\def\gr{\mathop{\hbox {gr}}\nolimits}
\DeclareMathOperator\Dim{Dim}

\def\dim{\mathop{\hbox {dim}}\nolimits}

\def\ch{\mathop{\hbox {ch}}\nolimits}

\def\Hom{\mathop{\hbox {Hom}}\nolimits}
\DeclareMathOperator\Ime{Im}
\DeclareMathOperator\im{Im}

\def\ker{\mathop{\hbox{Ker}}\nolimits}
\DeclareMathOperator\Ker{Ker}

\def\Spin{\mathop{\hbox {Spin}}\nolimits}
\def\Cas{\mathop{\hbox {Cas}}\nolimits}

\newcommand{\pf}{\begin{proof}}
\newcommand{\epf}{\end{proof}}
\newcommand{\eq}{\begin{equation}}
\newcommand{\eeq}{\end{equation}}
\newcommand{\eqn}{\begin{equation*}}
\newcommand{\eeqn}{\end{equation*}}

\newcommand{\fra}{\mathfrak{a}}
\newcommand{\frb}{\mathfrak{b}}

\newcommand{\frg}{\mathfrak{g}}
\newcommand{\frh}{\mathfrak{h}}
\newcommand{\frk}{\mathfrak{k}}
\newcommand{\frl}{\mathfrak{l}}
\newcommand{\frn}{\mathfrak{n}}
\newcommand{\fro}{\mathfrak{o}}
\newcommand{\frp}{\mathfrak{p}}
\newcommand{\frq}{\mathfrak{q}}

\newcommand{\frs}{\mathfrak{s}}
\newcommand{\frt}{\mathfrak{t}}

\newcommand{\fg}{\mathfrak{g}}

\newcommand{\bbC}{\mathbb{C}}

\newcommand{\bbN}{\mathbb{N}}

\newcommand{\bbR}{\mathbb{R}}

\newcommand{\bbZ}{\mathbb{Z}}
\newcommand{\caV}{\mathcal{V}}
\newcommand{\caO}{\mathcal{O}}

\newcommand{\be}{\begin{equation}}
\newcommand{\beu}{\begin{equation*}}
\newcommand{\defn}{\overset{\text{def.}}=}
\newcommand{\tr}{\operatorname{tr}}
\newcommand{\sgn}{\operatorname{sgn}}

\newtheorem{conj}[equation]{Conjecture}
\newtheorem{thm}[equation]{Theorem}
\newtheorem{cor}[equation]{Corollary}

\newtheorem{prop}[equation]{Proposition}

\newtheorem{ex}[equation]{Example}

\numberwithin{equation}{section}

\begin{document}
%\today

\bigskip
\title[Translation principle for Dirac index]{Translation principle for Dirac index}

\author{Salah Mehdi, Pavle Pand\v zi\'c and David Vogan}

\address{Institut Elie Cartan de Lorraine, CNRS - UMR 7502, Universit\'e de Lorraine, Metz, F-57045, France.}
\email{salah.mehdi@univ-lorraine.fr}
\address{Department of Mathematics, University of Zagreb, Zagreb, Croatia.}
\email{pandzic@math.hr}
\address{Department of Mathematics, Massachusetts Institute of Technology, Cambridge, MA 02139, U.S.A.}
\email{dav@math.mit.edu}

\thanks{The first and the second named authors were partially supported by a bilateral French-Croatian grant, PHC-Cogito 28353VL. 
The second named author was partially supported by a grant from Ministry of science, education and sport of Republic of Croatia. 
The third named author was supported in part by NSF grant DMS 0967272.}

\begin{abstract}
Let $G$ be a finite cover of a closed connected transpose-stable 
subgroup of $GL(n,\bR)$ with complexified Lie algebra $\frg$.     
Let $K$ be a maximal compact subgroup of $G$, and assume that $G$ and 
$K$ have equal rank. We prove a translation principle for the Dirac 
index of virtual $(\frg,K)$-modules. As a byproduct, to each coherent 
family of such modules, we attach a polynomial on the dual of the
compact Cartan subalgebra of $\frg$. This ``index polynomial''
generates an irreducible representation of the Weyl group contained in the coherent
continuation representation. We show that the index polynomial is the
exact analogue on the compact Cartan subgroup of King's character
polynomial. The character polynomial was defined in \cite{K1} on the
maximally split Cartan subgroup, and it was shown to be equal to the
Goldie rank polynomial up to a scalar multiple.  
In the case of representations of Gelfand-Kirillov dimension at most
half the dimension of $G/K$, we also conjecture an explicit
relationship between our index  
polynomial and the multiplicities of the irreducible components 
occuring in the associated cycle of the corresponding coherent family. 
\end{abstract}

\maketitle

%%%%%%%%%%%%%%%%%%%%%%%%%%%%%%%%%%%%%%%%%%%%%%%%%%%%%%%%%%%%%%%%%%%%%%%%%%%%%%%%
\section{Introduction}\label{section intro}

%DV
Let $G$ be a connected real reductive group; precisely, a finite cover
%DV
of a closed connected transpose-stable  
subgroup of $GL(n,\bR)$ with complexified Lie algebra $\frg$.   
Let $K$ be a maximal compact subgroup of $G$.  Write    
$\frg=\frk+\frp$ for the corresponding Cartan decomposition of $\frg$,
where $\frk$ is the complexified Lie algebra of $K$.    

\begin{subequations}\label{se:cohintro} 
Let $T\subset K$ be a maximal torus, so that $H_c = G^T$ is a
maximally compact Cartan subgroup, with 
% Let $H_{s}\subset G$ be a maximally split Cartan subgroup with complexified 
Lie algebra $\frh_{c}$.  
Let $\Lambda\subset\widehat{H_{c}}\subset\frh_{c}^{\star}$ be the
lattice of weights of finite-dimensional $(\frg,K)$-modules.  
For a fixed $\lambda_{0}\in\frh_{c}^{\star}$ regular, a family of
virtual $(\frg,K)$-modules $X_\lambda$,  
$\lambda\in\lambda_{0}+\Lambda$, is called {\it coherent} if  
for each $\lambda$, $X_\lambda$ has infinitesimal character $\lambda$,
and for any finite-dimensional $(\frg,K)$-module $F$, and for any $\lambda$,
%%%%%%%%%%%%%%
\begin{equation}\label{cohintro} 
X_\lambda\otimes F = \sum_{\mu\in\Delta(F)} X_{\lambda+\mu},
\end{equation}
%%%%%%%%%%%%%%%
where $\Delta(F)$ denotes the multiset of all weights of $F$. (A more
complete discussion appears in Section \ref{section coherent}.) The
reason for studying coherent families is that if $X$ is any
%DV
irreducible $(\frg,K)$-module of infinitesimal character $\lambda_0$,
then there is a {\it unique} coherent 
%DV
family with the property that
%%%%%%%%%%%%%%%%%%%%%%
\begin{equation}\label{existscoherent}
X_{\lambda_0} = X.
\end{equation}
%%%%%%%%%%%%%%%%%%%%%
For any invariant of Harish-Chandra modules, one can therefore
ask how the invariant of $X_\lambda$ changes with $\lambda \in
\lambda_0 + \Lambda$.  The nature of this dependence is then a new
invariant of $X$. This idea is facilitated by the fact that
\begin{equation}
\text{$X_\lambda$ is irreducible or zero whenever $\lambda$ is
  integrally dominant;}
\end{equation}
zero is possible only for singular $\lambda$. (See for example
\cite{V2}, sections 7.2 and 7.3.) The notion of ``integrally
dominant'' is recalled in \eqref{intdom}; we write
$(\lambda_0+\Lambda)^+$ for the cone of integrally dominant elements.
% A virtual $(\frg,K)$-module $X$ with regular infinitesimal character 
% $\lambda_{0}\in\frh_{c}^{\star}$ can be placed in a unique coherent 
% family as above (see Theorem 7.2.7 in \cite{V2}, and the references 
% therein).  Using this, one can define an action of the integral Weyl 
% group $W(\lambda_{0})$ attached to $\lambda_{0}$ on the set of virtual
% $(\frg,K)$-modules with infinitesimal character $\lambda_{0}$. A 
% decomposition into irreducible components of this 
% $W(\lambda_{0})$-representation, known as the coherent continuation 
% representation, was obtained by Barbasch and Vogan (see \cite{BV1b}). 
% The study of coherent continuation representations is important for 
% deeper understanding of coherent families.  
We may therefore define 
\begin{equation}\label{annintro}
\Ann(X_{\lambda}) = \text{annihilator in $U(\frg)$ of $X_\lambda$}
\qquad (\lambda \in (\lambda_0 + \Lambda)^+).
\end{equation}
The ideal $\Ann(X_\lambda)$ is primitive if $X_\lambda$ is
irreducible, and equal to $U(\frg)$ if $X_\lambda = 0$. Write
$\rk(U(\frg)/\Ann(X_{\lambda})$) for the Goldie rank of the %quotient 
algebra $U(\frg)/\Ann(X_{\lambda})$. Let $W_\frg$ be the Weyl group of
$\frg$ with respect to $\frh_c$.  Joseph proved that the $\bN$-valued
map defined on integrally dominant $\lambda \in \lambda_{0}+\Lambda$ by
\begin{equation}\label{goldieintro} 
\lambda\mapsto \rk(U(\frg)/\Ann(X_{\lambda})) 
\end{equation}  
extends to a $W_{\frg}$-harmonic polynomial $P_{X}$ on
$\frh_{c}^{\star}$ called the {\it Goldie rank
  polynomial} for $X$. The polynomial $P_X$ is homogeneous of degree
$\sharp R_{\frg}^{+}-\Dim(X)$, where $\sharp R_{\frg}^{+}$ denotes the
number of positive $\frh_{c}$-roots in $\frg$ and $\Dim(X)$ is the
Gelfand-Kirillov dimension of $X$. 
Moreover, $P_{X}$ % can be expressed in terms
% of Kazhdan-Lusztig 
% polynomials evaluated at $1$ and 
generates an irreducible representation of $W_{\frg}$. See \cite{J1I},
\cite{J1II} and \cite{J1III}. 

% There is a representation of $W(\lambda_{0})$ on a certain subspace of
% the space of polynomials on $\frh_s^*$, defined as follows. 

There is an interpretation of the $W_{\frg}$-representation generated
by $P_X$ in terms of the Springer
correspondence. % Suppose that the modules  
% $X_{\lambda}$ are such that each $\Ann(X_{\lambda})$ is a primitive
% ideal in $U(\frg)$. 
For all  $\lambda\in (\lambda_0+\Lambda)^+$ such that
$X_\lambda \ne 0$ (so for example for all 
integrally dominant regular $\lambda$), the associated variety  
${\mathcal V}(\gr(\Ann(X_{\lambda})))$  
(defined by the associated graded ideal of $\Ann(X_{\lambda})$, in the
symmetric algebra $S(\frg)$) is the Zariski closure of a single nilpotent
$G_{\bC}$-orbit ${\mathcal O}$ in  
$\frg^{\star}$, independent of $\lambda$. % {\clrr 
(Here $G_\bC$ is a connected complex reductive algebraic
  group having Lie algebra $\frg$.) % } 
Barbasch and Vogan
proved that the Springer representation of $W_{\frg}$ attached to
${\mathcal O}$ coincides with the $W_{\frg}$-representation generated
by the Goldie rank polynomial $P_{X}$ (see \cite{BV1}).  

Here is another algebro-geometric interpretation of $P_{X}$. Write
\begin{equation}\label{eq:Korbit}
{\mathcal O}\cap (\frg/\frk)^* = \coprod_{j=1}^r {\mathcal O}^j
\end{equation}
for the decomposition into (finitely many) orbits of $K_{\mathbb C}$. 
%DV \textcolor{blue}{
(Here $K_{\mathbb C}$ is the complexification of $K$.)
%DV }
Then the associated cycle of each $X_\lambda$ is
\begin{equation}\label{multintro}
\Ass(X_\lambda) = \coprod_{j=1}^r m^j_X(\lambda) \overline{{\mathcal
    O}^j} \qquad (\lambda \in (\lambda_0 + \Lambda)^+) % \ \text{integrally
%  dominant})
\end{equation}
(see Definition 2.4, Theorem 2.13, and Corollary 5.20 in
\cite{V3}). The component multiplicity $m^j_X(\lambda)$ is a
function taking nonnegative integer values, and  extends to a
polynomial function on $\frh_{c}^*$. We call this polynomial the {\it
  multiplicity polynomial} for $X$ on the orbit ${\mathcal O}^j$. The
connection with the Goldie rank polynomial is that each $m^j_X$ is a
scalar multiple of $P_X$; this is a consequence of the proof of
Theorem 5.7 in \cite{J1II}. 

On the other hand, Goldie rank polynomials can be interpreted in terms
of the asymptotics of the global character $\ch_{\frg}(X_{\lambda})$
of $X_{\lambda}$ on a maximally split Cartan subgroup $H_{s} \subset
G$ with Lie algebra $\frh_{s,0}$.  Namely, if % $\rho_{\frg}$ denotes half 
% the sum of positive $\frh_{\frs}$-roots, 
$x \in \frh_{s,0}$ is a generic regular element, King proved that the map  
\begin{equation}\label{kingintro} 
\lambda\mapsto \lim_{t\rightarrow
  0}t^{\Dim(X)}\ch_{\frg}(X_{\lambda})(\exp tx)  
\end{equation}  
on $\lambda_{0}+\Lambda$ extends to a polynomial $C_{X,x}$ 
%DV
on $\frh^{\star}_{c}$. We call this polynomial {\it King's character 
% Because you defined the coherent family on h_c (unlike King), the
% polynomial you define is on h_c^* 
polynomial}. It coincides with the Goldie rank polynomial $P_{X}$ up to
a constant factor depending on $x$ (see \cite{K1}). 
% {\clrr (Here we view $P_X$ as
% a polynomial on $\frh_s^*$ via a Cayley transform.)} 
%DV
More precisely, as
a consequence of \cite{SV}, one can show that there is a formula
\begin{equation}\label{eq:multchar}
C_{X,x} = \sum_{j=1}^r a^j_x m^j_{X};
\end{equation}
the constants $a^j_x$ are independent of $X$, and this formula is
valid for any $(\frg,K)$-module whose annihilator has associated
variety contained in $\overline{\mathcal O}$.
The polynomial $C_{X,x}$ expresses 
the dependence on $\lambda$ of the leading term in the Taylor
expansion of the numerator of the character of $X_\lambda$ on the
maximally split Cartan $H_{s}$.  

\end{subequations} %se:cohintro

In this paper, we assume that $G$ and $K$ have equal rank.  Under this
assumption, we use Dirac index to obtain the analog of King's
asymptotic character formula (\ref{kingintro}), or equivalently of the
Goldie rank polynomial (\ref{goldieintro}), in the case when $H_s$ is
replaced by a compact Cartan subgroup $T$ of $G$.  In the course of
doing this, we first prove a translation principle for the Dirac
index.

To define the notions of Dirac cohomology and index, we first recall
that there is a {\it Dirac operator} 
$D\in U(\frg)\otimes C(\frp)$, where $C(\frp)$ is the Clifford algebra
of $\frp$ with respect to an invariant non-degenerate symmetric  
bilinear form $B$ (see Section \ref{section setting}). % The Dirac
% operator is defined as % the quotient 
% \[
% D=\sum_i b_i\otimes d_i,
% \]
% where $b_i$ is any basis of $\frp$, and $d_i$ is the dual basis with
% respect to $B$. It is independent of the choice of the basis $b_i$,
% and $K$-invariant for the tensor product of adjoint actions on the factors. 
If $S$ is a spin module for $C(\frp)$ then $D$ acts  
on $Y\otimes S$ for any $(\frg,K)$-module $Y$. The {\it Dirac cohomology} $H_{D}(Y)$ of $Y$ is defined as 
\beu
H_{D}(Y)=\Ker D / \Ker D\cap\Ime D.
\end{equation*}
It is a module for the spin double cover $\Kt$ of $K$.
Dirac cohomology was introduced by Vogan in the late 1990s (see
\cite{V4}) and turned out to be an interesting invariant attached to
$(\frg,K)$-modules (see \cite{HP2} for a thorough discussion).  

We would now like to study how Dirac cohomology varies over a coherent
family. This is however not possible; since Dirac cohomology is not an
exact functor, it cannot be defined for virtual  
$(\frg,K)$-modules. To fix this problem, we will replace Dirac cohomology by the Dirac index. 
(We note that there is a relationship between Dirac cohomology and
translation functors; see %DV {\clrr
\cite{MP}, \cite{MP08}, \cite{MP09},
\cite{MP10}.) 

\begin{subequations}\label{se:cohindex}
Let $\frt$ be the complexified Lie algebra of the compact Cartan
subgroup $T$ of $G$. Then $\frt$ is a Cartan subalgebra of both $\frg$
and $\frk$. In this case, the spin module $S$ for $\widetilde{K}$ is
the direct sum of two pieces $S^+$ and $S^-$, and the Dirac cohomology
$H_D(Y)$ breaks up accordingly into $H_D(Y)^+$ and $H_D(Y)^-$. %DV{\clrr 
If $Y$ is admissible and has infinitesimal character, %DV} 
define the {\it Dirac index of $Y$} to be the virtual $\widetilde{K}$-module  
\begin{equation}
I(Y)= H_D(Y)^+-H_D(Y)^-.
\end{equation}
%DV {\clrr 
This definition can be extended to arbitrary finite length modules
(not necessarily with infinitesimal character),  replacing $H_D$ by
the higher Dirac cohomology of \cite{PS}. See Section \ref{section
  index}. Then $I$, considered as a functor from finite length
$(\frg,K)$-modules to virtual $\Kt$-modules, is additive with respect
to short exact sequences (see Lemma 
\ref{exact} and the discussion below (\ref{index formula})), so it
makes sense also for virtual $(\frg,K)$-modules. 
%DV }
Furthermore, $I$ satisfies the  
following property (Proposition \ref{main}): for any
finite-dimensional $(\frg,K)$-module $F$,  
\beu
I(Y\otimes F)=I(Y)\otimes F.
\end{equation*}

Let now $\{X_{\lambda}\}_{\lambda \in \lambda_{0}+\Lambda}$ be a
coherent family 
of virtual $(\frg,K)$-modules. % based on $\lambda_{0}\in\frt^{\star}$
% and the lattice $\Lambda\subset\widehat{T}\subset\frt^{\star}$ of
% weights of finite-dimensional $(\frg,K)$-modules. 

By a theorem of Huang and Pand{\v{z}}i{\'c}, the $\frk$-infinitesimal
character of any $\Kt$-type contributing to the Dirac cohomology
$H_D(Y)$ of an irreducible $(\frg,K)$-module $Y$ is
$W_{\frg}$-conjugate to the $\frg$-infinitesimal character of $Y$ (see
Theorem \ref{HPmain}). In terms of the virtual representations $\Et$
of $\Kt$ defined in Section \ref{section coherent}, the conclusion is that we
may write
\begin{equation}\label{indexformula}
I(X_{\lambda_{0}})=\sum_{w\in W_\frg} a_w \Et_{w\lambda_{0}}
\end{equation}
%
% where $\Et$ denotes the coherent family of finite-dimensional
% $\Kt$-modules and 
with $a_w\in \bZ$. Then, for any $\nu\in\Lambda$,
we have (Theorem \ref{translindex}): 
\eq
\label{indextransintro}
I(X_{\lambda_{0}+\nu})=\sum_{w\in W_\frg} a_w \Et_{w(\lambda_{0}+\nu)}
\eeq
with the same coefficients $a_w$. It follows that
$I(X_{\lambda_{0}})\neq 0$ implies  
$I(X_{\lambda_{0}+\nu})\neq 0$, provided both $\lambda_{0}$ and
$\lambda_{0}+\nu$ are regular for $\frg$ (Corollary \ref{nonzeroindex}). 

Combining the translation principle for Dirac index
\eqref{indextransintro} with the Weyl dimension formula for $\frk$, we
conclude that the map  
\begin{equation}\label{indexintro}
\lambda_{0}+\Lambda\rightarrow\bZ,\qquad \lambda\mapsto\dim I(X_{\lambda})
\end{equation}
extends to a $W_{\frg}$-harmonic polynomial $Q_{X}$ on $\frt^{\star}$
(see Section \ref{section Weyl group}). We call the polynomial $Q_X$
the {\it index polynomial} attached to $X$ and $\lambda_{0}$. If $Q_X$ is
nonzero, its degree is equal to the number $R_{\frk}^{+}$ of
positive $\frt$-roots in $\frk$. More precisely, $Q_X$ belongs to the
irreducible representation of $W_{\frg}$ generated by the Weyl
dimension formula for $\frk$ (Proposition
\ref{harmonic}). Furthermore, the coherent continuation representation
generated by $X$ must contain a copy of the index polynomial
representation (Proposition \ref{wequi}). We also prove that the index
polynomial vanishes for small representations. Namely, 
if the Gelfand-Kirillov dimension $\Dim(X)$ is less than the number 
$\sharp R_{\frg}^{+}-\sharp R_{\frk}^{+}$ of positive noncompact
$\frt$-roots in $\frg$, then $Q_{X}=0$ (Proposition \ref{indexzero}). 

An important feature of the index polynomial is the fact that $Q_{X}$
is the exact analogue of King's character polynomial
(\ref{kingintro}), but attached to the character on the compact Cartan
subgroup instead of the maximally split Cartan subgroup (see Section
\ref{section Goldie rank}).  In fact, $Q_{X}$ expresses  
the dependence on $\lambda$ of the (possibly zero) leading term in the Taylor
expansion of the numerator of the character of $X_\lambda$ on the
compact Cartan $T$:   
% to a constant  
% equal to the leading term in the Taylor expansion of the global
% character on the compact Cartan subgroup (Theorem \ref{ind=char}):
for any $y\in \frt_0$ regular, we have
\beu
\lim_{t\to 0+} t^{\sharp R_{\frg}^{+}-\sharp R_{\frk}^{+}}
\ch_\frg(X_\lambda)(\exp ty)=\textstyle{\frac{\prod_{\alpha\in
      R_\frk^+}\alpha(y)}{\prod_{\alpha\in
      R_\frg^+}\alpha(y)}}\, Q_X(\lambda). 
\end{equation*}

In particular, if $G$ is semisimple of Hermitian type, and if $X$ is
the $(\frg,K)$-module of a holomorphic discrete series representation,
then the index polynomial $Q_{X}$ 
coincides, up to a scalar multiple, with the Goldie rank polynomial
$P_{X}$ (Proposition \ref{holods}). Moreover, if $X$ is the $(\frg,K)$-module of any
discrete series representation (for $G$ not
necessarily Hermitian), then $Q_X$ and $P_X$ are both
divisible by the product of linear factors corresponding to the roots
generated by the $\tau$-invariant of $X$ (Proposition \ref{tau}).  
Recall that the $\tau$-invariant of the $(\frg,K)$-module $X$ consists
of the simple roots $\alpha$ such that the translate of $X$ to the
wall defined by $\alpha$ is zero (see Section 4 in \cite{V1} or
Chapter 7 in \cite{V2}). 

Recall the formula \eqref{eq:multchar} relating King's character
polynomial to the multiplicity polynomials for the associated
cycle. In Section 7, we conjecture a parallel relationship between the
index polynomial $Q_X$ and the multiplicity polynomials. For that, we
must assume that the $W_{\frg}$-representation generated by the Weyl dimension 
formula for $\frk$ corresponds to a nilpotent $G_{\bC}$-orbit
${\mathcal O}_{K}$ via the Springer correspondence. 
(At the end of Section \ref{orbits}, we list the classical groups for which
this assumption is satisfied.)
% Denote the $K_\bbC$-components of
% $\overline{{\mathcal O}_{K}}\cap (\frg/\frk)^{\star}$ by
% $\overline{{\mathcal%  O}_{K}^{j}}$. 
% For any Harish-Chandra module $X$ for $G$ satisfying 
% ${\mathcal V}(\gr(\Ann(X)))=\overline{O_{K}}$, one can express the
% Goldie rank polynomial attached to $X$ as 
%
% \beu
% P_{X}=\sum_{j}a_{j,X}m_{\{X\}}^{j},
% \end{equation*}
%
% where $m_{\{X\}}^{j}(\lambda)=m^j_{X_{\lambda}}$ denotes the
% multiplicity of % $\overline{{\mathcal O}_{K}^{j}}$ in the 
% associated cycle of $X_{\lambda}$, and $a_{j,X}$ are complex numbers
% depending on $j$ and $X$ \textcolor{red}{(see, for instance,
% \cite{BV1}, \cite{Ch},  \cite{J2}, \cite{SV})}. 
Then we conjecture (Conjecture \ref{conj}): if ${\mathcal V}(\gr(\Ann(X)))\subset \overline{{\mathcal O}_K}$, then
\begin{equation}\label{eq:multindex}
Q_{X}=\sum_{j}c^j m_{X}^{j}.
\end{equation}
Here the point is that the coefficients $c^j$ should be integers independent of
$X$.  We check that this conjecture holds in the case when $G=SL(2,\bR)$ and also
when $G=SU(1,n)$ with $n\geq 2$.

In the following we give a few remarks related to the significance of the above conjecture.

Associated varieties are a beautiful and concrete invariant for
representations, but they are too crude to
distinguish representations well. For example, all holomorphic discrete series
have the same associated variety. Goldie rank
polynomials and multiplicity functions both offer more information, but
the information is somewhat
difficult to compute and to interpret precisely.  The index polynomial
is easier to compute and interpret precisely; it can be computed from
knowing the restriction to $K$, and conversely, it contains 
fairly concrete information about the restriction to $K$. In the setting of (\ref{eq:multindex})
(that is, for fairly small representations), the conjecture
says that the index polynomial should be built from multiplicity
polynomials in a very simple way.  

The conjecture implies that, for these small representations,  the
index polynomial must be a multiple of the Goldie rank polynomial.
This follows from the fact that each $m^i_X$ is a multiple of $P_X$,
mentioned below (\ref{multintro}).
The interesting thing about this is that the index polynomial is perfectly
well-defined for {\it larger} representations as well. In some sense it is
defining something like ``multiplicities" for $\caO_K$ even when $\caO_K$ is not
a leading term. 

This is analogous to a result of Barbasch, which says that one can
define for any character expansion a number that is the multiplicity
of the zero orbit for finite-dimensional representations. In the case
of discrete series, this number turns out to be the formal degree (and
so is something really interesting). This indicates
that the index polynomial is an example of an interesting ``lower order
term" in a character expansion. We can hope that a fuller
understanding of all such lower order terms could be a path to
extending the theory of associated varieties to a more complete
invariant of representations.

\end{subequations} %{se:cohindex}
%%%%%%%%%%%%%%%%%%%%%%%%%%%%%%%%%%%%%%%%%%%%%%%%%%%%%%%%%%%%%%%%%%%%%%%%%%%%%%%%
\section{Setting}\label{section setting}

Let $G$ be a finite cover of a closed connected transpose-stable
subgroup of $GL(n,\bR)$, with Lie algebra $\frg_{0}$. We denote by
$\Theta$ the Cartan involution of $G$ corresponding to the usual
Cartan involution of $GL(n,\bR)$ (the transpose inverse). Then
$K=G^\Theta$ is a maximal compact subgroup of $G$. Let
$\frg_{0}=\frk_{0}\oplus\frp_{0}$ be the Cartan decomposition of
$\frg_{0}$, with $\frk_0$ the Lie algebra of $K$.  
Let $B$ be the trace form on $\frg_{0}$. Then $B$ is positive definite
on $\frp_{0}$ and negative definite on $\frk_{0}$, and $\frp_{0}$ is
the orthogonal of $\frk_{0}$ with respect to $B$. We shall drop the
subscript `0' on real vector spaces to denote their
complexifications. Thus $\frg=\frk\oplus\frp$ denotes the Cartan
decomposition of the complexified Lie algebra of $G$. The linear
extension of $B$ to $\frg$ is again denoted by $B$. % \textcolor{red}{
Let $G_{\bC}$ be a connected reductive algebraic group over $\bC$ with Lie
algebra $\fg$. % such that $G$ is a real form of $G_{\bC}$.}
  
Let $C(\frp)$ be the Clifford algebra of $\frp$ with respect to $B$
and let $U(\frg)$ be the universal enveloping algebra of $\frg$. 

The Dirac operator $D$ is defined as
\[
D=\sum_i b_i\otimes d_i\in U(\frg)\otimes C(\frp),
\]
where $b_i$ is any basis of $\frp$ and $d_i$ is the dual basis with
respect to $B$. Then $D$ is independent of the 
choice of the basis $b_i$ and is $K$-invariant. Moreover, the square
of $D$ is given by the following formula due to  
Parthasarathy \cite{P}:
\eq
\label{Dsquared}
D^2=-(\Cas_\frg\otimes 1+\|\rho_\frg\|^2)+(\Cas_{\frk_\Delta}+\|\rho_\frk\|^2).
\eeq
Here $\Cas_\frg$ is the Casimir element of $U(\frg)$ and $\Cas_{\frk_\Delta}$ is the Casimir element of $U(\frk_\Delta)$,
where $\frk_\Delta$ is the diagonal copy of $\frk$ in $U(\frg)\otimes C(\frp)$, defined using the obvious embedding
$\frk\hookrightarrow U(\frg)$ and the usual map $\frk\to\frs\fro(\frp)\to C(\frp)$. See \cite{HP2} for details.

If $X$ is a $(\frg,K)$-module, then $D$ acts on $X\otimes S$, where $S$ is a spin module for $C(\frp)$. The Dirac cohomology of $X$ is the module
\[
H_D(X)=\ker D / \Ker D\cap\Ime D
\]
for the spin double cover $\Kt$ of $K$. If $X$ is unitary or finite-dimensional, then
\[
H_D(X)=\Ker D=\Ker D^2.
\]
The following result of \cite{HP1} was conjectured by Vogan
\cite{V3}. Let $\frh=\frt\oplus\fra$ be a fundamental Cartan
subalgebra of $\frg$. We view $\frt^*\subset\frh^*$ by extending
functionals on $\frt$ by 0 over $\fra$. Denote by $R_{\frg}$ (resp. $R_{\frk}$)
the set of $(\frg,\frh)$-roots (resp. $(\frk,\frt)$-roots). We fix compatible 
positive root systems $R^{+}_{\frg}$ and  $R^{+}_{\frk}$ for $R_\frg$ and $R_\frk$ respectively.
In particular, this determines the half-sums of positive roots $\rho_\frg$ and $\rho_\frk$ as usual. Write $W_{\frg}$ (resp. $W_{\frk}$) for the
Weyl group associated with $(\frg,\frh)$-roots
(resp. $(\frk,\frt)$-roots). 

\begin{thm}
\label{HPmain}
Let $X$ be a $(\frg,K)$-module with infinitesimal character corresponding to $\Lambda\in\frh^*$ via the Harish-Chandra isomorphism. Assume
that $H_D(X)$ contains the irreducible $\Kt$-module $E_\gamma$ with highest weight $\gamma\in\frt^*$. 
Then $\Lambda$ is equal to $\gamma+\rho_\frk$ up to conjugation by the Weyl group $W_\frg$. In other words, the $\frk$-infinitesimal character of any $\Kt$-type contributing to $H_D(X)$ 
is $W_\frg$-conjugate to the $\frg$-infinitesimal character of $X$.
\end{thm}

%%%%%%%%%%%%%%%%%%%%%%%%%%%%%%%%%%%%%%%%%%%%%%%%%%%%%%%%%%%%%%%%%%%%%%%%%%%%%%%%
\section{Dirac index}
\label{section index}

Throughout the paper we assume that $\frg$ and $\frk$ have equal rank, 
i.e., that there is a compact Cartan subalgebra $\frh=\frt$ in $\frg$. 
In this case,  $\frp$ is even-dimensional, so (as long as $\frp \neq\{0\}$)  
the spin module $S$ for the spin group $\Spin(\frp)$ (and therefore
for $\Kt$) is the direct sum of two pieces, which we will call $S^+$  
and $S^-$. To say which is which, it is enough to choose an
$SO(\frp)$-orbit of maximal isotropic subspaces  
of $\frp$. We will sometimes make such a choice by fixing a positive
root system $\Delta^+(\frg,\frt)$ for $\frt$ in $\frg$,  
and writing $\frn=\frn_\frk + \frn_\frp$ for the corresponding sum of
positive root spaces. Then $\frn_\frp$ is a choice of  
maximal isotropic subspace of $\frp$.  

%DV {\clrr 
The full spin module % corresponding to 
may be realized using $\frn_\frp$ as $S \simeq \bigwedge\frn_\frp$,
with the $C(\frp)$-action 
defined so that elements of $\frn_\frp$ act by wedging, and elements of
the dual isotropic space $\frn_\frp^-$ corresponding to the negative roots
act by contracting. (Details may be found for example in \cite{Chev}
at the beginning of Chapter 3.) In particular, the action of $C(\frp)$
respects parity of degrees: odd elements of $C(\frp)$ carry
$\bigwedge^{\text{even}}\frn_\frp$ to
$\bigwedge^{\text{odd}}\frn_\frp$ and so on. Because $\Spin(\frp)
\subset C_{\text{even}}(\frp)$, it follows that $\Spin(\frp)$
preserves the decomposition 
\[
S \simeq \bigwedge \frn_\frp = \bigwedge\nolimits^{\!\!\text{even}}\frn_\frp \oplus
\bigwedge\nolimits^{\!\!\text{odd}}\frn_\frp \defn S^+ \oplus S^-.
\]

The group $\Kt$ acts on $S$ as usual, through the map
$\Kt\to\Spin(\frp)\subset C(\frp)$, and hence also the Lie algebra
$\frk$ acts, through the map $\alpha:\frk\to
\frs\fro(\frp)\hookrightarrow C(\frp)$.  We call these actions of
$\Kt$ and $\frk$ the spin actions. It should however be noted that
although we wrote $S\simeq \bigwedge\frn_\frp$, the $\frt$-weights of
$S$ for the spin action are not the weights of $\bigwedge\frn_\frp$,
i.e., the sums of distinct roots in $\frn_\frp$, but rather these
weights shifted by $-(\rho_\frg-\rho_\frk)$. This difference comes
from the construction of the map $\alpha$ and the action of $C(\frp)$
on $S$.

% Having said all this, we can now define
% $S^+=\bigwedge^{\text{even}}\frn_\frp$. 
In particular, the weights of $S^+ \simeq
\bigwedge^{\text{even}}\frn_\frp $ are
% is the span of all the weight spaces of $\frt$ in $S$ of the form
\[
-\rho_\frg + \rho_\frk + \text{(sum of an even number of  
distinct roots in } \frn_\frp).
\]
Similarly, the weights of $S^- \simeq \bigwedge^{\text{odd}}\frn_\frp$ are
% with $\frt$-weights of the form
\[
-\rho_\frg + \rho_\frk + \text{(sum of an odd number of  
distinct roots in } \frn_\frp).
\]
% DV }  

% this needs to be explained,
% and "=" is not appropriate; there are actions of $\frt$ on both
% sides and they are not the same. Actually my guess is that any reasonable
% explanation will lead instead to
% $-\rho_\frg + \rho_\frk + \text{(sum of even number of  
% distinct roots in $\frn_\frp$)}$. 

The Dirac operator $D$ 
interchanges $X\otimes S^+$ and $X\otimes S^-$ for any $(\frg,K)$-module $X$. 
(That is because it is of degree 1 in  
the Clifford factor.) It follows that the Dirac cohomology $H_D(X)$
also breaks up into even and odd parts, which we  
denote by $H_D(X)^+$ and $H_D(X)^-$ respectively. If $X$ is of finite length, then $H_D(X)$ is  
finite-dimensional, as follows from (\ref{Dsquared}), which implies
that $\Ker D^2$ is finite-dimensional for any  
admissible module $X$. 
%DV {\clrr 
If $X$ is of finite length and has infinitesimal character, then we define the 
 Dirac index of $X$ as the virtual
$\Kt$-module  
%DV }
\eq
\label{defindex wrong}
I(X)= H_D(X)^+-H_D(X)^-.
\eeq
The first simple but important fact is the following proposition,
which is well known for the case of discrete series or finite-dimensional modules.

\begin{prop}
\label{propindex}
%DV {\clrr 
Let $X$ be a finite length $(\frg,K)$-module with infinitesimal character. Then there is an equality of virtual $\Kt$-modules
\[
X\otimes S^+ - X\otimes S^- = I(X).
\]
%DV }
\end{prop}
\pf By Parthasarathy's formula for $D^2$ (\ref{Dsquared}), 
$X\otimes S$ breaks into a direct sum of eigenspaces for $D^2$:
\[
X\otimes S=\sum_\lambda (X\otimes S)_\lambda.
\]
Since $D^2$ is even in the Clifford factor, this decomposition is compatible with the decomposition
into even and odd parts, i.e.,
\[
(X\otimes S)_\lambda=(X\otimes S^+)_\lambda \oplus (X\otimes S^-)_\lambda,
\]
for any eigenvalue $\lambda$ of $D^2$. Since $D$ commutes with $D^2$, it preserves each eigenspace.
Since $D$ also switches parity, we see that $D$ defines maps
\[
D_\lambda:(X\otimes S^{\pm})_\lambda \to (X\otimes S^{\mp})_\lambda
\]
for each $\lambda$. If $\lambda\neq 0$, then $D_\lambda$ is clearly an isomorphism (with inverse
$\frac{1}{\lambda}D_\lambda$), and hence
\[
X\otimes S^+ - X\otimes S^- = (X\otimes S^+)_0 - (X\otimes S^-)_0.
\]
Since $D$ is a differential on $\Ker D^2$, and the cohomology of this differential is exactly $H_D(X)$,
the statement now follows from the Euler-Poincar\'e principle.
\epf

\begin{cor}
\label{exact}
%DV {\clrr 
Let 
\[
0\to U\to V\to W\to 0
\]
be a short exact sequence of finite length $(\frg,K)$-modules, and assume that $V$ 
has infinitesimal character (so that $U$ and $W$ must have the same infinitesimal character as $V$).
Then there is an equality of virtual $\Kt$-modules
\[
I(V) = I(U) + I(W).
\] %DV }
\end{cor}
\pf %DV {\clrr 
This follows from the formula in Proposition \ref{propindex}, since
the left hand side of that formula  
clearly satisfies the additivity property. %DV }
\epf
% DV {\color{blue}
%{\color{blue} 
%The rest of this section was shortened from 2 pages to 1 page,
%following the request of the referee. You may take a look, but it is less essential than the other places.
%}
To study the translation principle, we need to deal with modules
$X\otimes F$, where $X$ is a finite length  
$(\frg,K)$-module, and $F$ is a finite-dimensional
$(\frg,K)$-module. Therefore, Proposition 
\ref{propindex} and Corollary \ref{exact} are not sufficient for our
purposes, because they apply only to modules with infinitesimal character. 
Namely, if $X$ is of finite length and has infinitesimal character,
then $X\otimes F$ is of finite length, but it 
typically cannot be written as a direct sum of modules with
infinitesimal character. Rather, some of the summands of 
$X\otimes F$ only have generalized infinitesimal character. Recall
that $\chi:Z(\frg)\to\bbC$ is the generalized 
infinitesimal character of a $\frg$-module $V$ if there is a positive
integer $N$ such that 
\[
(z-\chi(z))^N=0\quad\text{on }V,\qquad \text{for every }z\in Z(\frg),
\]
where $Z(\frg)$ denotes the center of $U(\frg)$. Here is an example
showing that Proposition \ref{propindex} and Corollary \ref{exact} can
fail for modules with generalized infinitesimal character. 

\begin{ex}[\cite{PS}, Section 2] {\rm Let $G=SU(1,1)\cong SL(2,\bbR)$,
so that $K=S(U(1)\times U(1))\cong U(1)$, and $\frg=\frs\frl(2,\bbC)$.
Then there is an indecomposable $(\frg,K)$-module $P$ fitting into the short exact sequence
\[
0\to V_0\to P\to V_{-2}\to 0,
\]
where $V_0$ is the (reducible) Verma module with highest weight 0, and
$V_{-2}$ is the (irreducible) Verma module 
with highest weight -2. One can describe the $\frg$-action on $P$ very explicitly,
and see that $\Cas_\frg$ does not act by a scalar on $P$, so $P$ does not have
infinitesimal character. 

Using calculations similar to \cite{HP2}, 9.6.5, one checks that
for the index defined by (\ref{defindex wrong}) the following holds:
\eq
\label{indexP}
I(P)=-\bbC_1;\qquad I(V_0)=-\bbC_1;\qquad I(V_{-2})=-\bbC_{-1},
\eeq 
where $\bbC_1$ respectively $\bbC_{-1}$ is the one-dimensional $\Kt$-module of weight $1$ respectively $-1$.
So Corollary \ref{exact} fails for $P$. It follows that Proposition \ref{propindex} must also fail. This can
also be seen directly, by computing $P\otimes S^+-P\otimes S^-$. 
}
\end{ex}

The reason for the failure of both Proposition \ref{propindex} and
Corollary \ref{exact} is the fact that the generalized 0-eigenspace for $D$
contains two Jordan blocks for $D$, one of length 1 and the other of length 3.
The block of length 3 does contribute to $P\otimes S^+-P\otimes S^-$, but not
to $I(P)$. With this in mind, a modified version of Dirac cohomology, called
``higher Dirac cohomology", has been recently defined by Pand\v zi\'c and Somberg \cite{PS}.
It is defined as $H(X)=\bigoplus_{k\in\bbZ_+} H^k(V)$, where
\[
H^k(V) = \im D^{2k}\cap \ker D \big/ \im D^{2k+1}\cap\ker D.
\]
For a module $X$ with infinitesimal character, $H(X)$ is the same as
$H_D(X)$; in general, $H(X)$ contains $H_D(X)=H^0(X)$. If $X$ is an
arbitrary finite length module, then $H(X)$ is composed from contributions
from all odd length Jordan blocks in the generalized 0-eigenspace for $D$. 
It follows that if we let $H(X)^\pm$ be the even and odd parts
of $H(X)$, and define the stable index as 
\eq
\label{defindex right}
I(X)=H(X)^+-H(X)^-,
\eeq 
then Proposition \ref{propindex} holds for any module $X$ of
finite length, i.e.,  
\eq
\label{index formula}
I(X)= X\otimes S^+-X\otimes S^-
\eeq
(\cite{PS}, Theorem 3.4). It follows that the index defined in this
way is additive with respect to short exact sequences (\cite{PS},  
 Corollary 3.5), and it therefore makes sense for virtual
 $(\frg,K)$-modules, i.e., it is well defined on the Grothendieck
 group of the Abelian category of finite length
 $(\frg,K)$-modules. Let us also mention that there is an analogue of
 Theorem \ref{HPmain} for $H(X)$ (\cite{PS}, Theorem 3.3.) 

There is another way to define the index that circumvents completely
the discussion of defining Dirac cohomology in the right way. Namely,
one can simply use the statement of Proposition \ref{propindex}, or
(\ref{index formula}), as the definition of the index $I(X)$. It is
clear that with such a definition the index does make sense for
virtual $(\frg,K)$-modules. Moreover, one shows as above that all of
the eigenspaces for $D^2$ for nonzero eigenvalues cancel out in
(\ref{index formula}), so what is left is a finite combination of
$\Kt$-types, appearing in the 0-eigenspace for $D^2$. 

Whichever of these two ways to define $I(X)$ we take, we will from now
on work with Dirac index $I(X)$, defined for any virtual
$(\frg,K)$-module $X$, and satisfying (\ref{index formula}). 

%DV }

%%%%%%%%%%%%%%%%%%%%%%%%%%%%%%%%%%%%%%%%%%%%%%%%%%%%%%%%%%%%%%%%%%%%%%%%%
\section{Coherent families}
\label{section coherent}

Fix $\lambda_0\in\frt^*$ regular and let $T$ be a compact Cartan subgroup of $G$ with complexified Lie algebra $\frt$. 
We denote by $\Lambda\subset\widehat{T}\subset\frt^*$ the lattice of weights of finite-dimensional representations of $G$ 
(equivalently, of finite-dimensional $(\frg,K)$-modules). A family of virtual $(\frg,K)$-modules $X_\lambda$, 
$\lambda\in\lambda_{0}+\Lambda$, is called {\it coherent} if

\begin{enumerate}
\item $X_\lambda$ has infinitesimal character $\lambda$; and
\item for any finite-dimensional $(\frg,K)$-module $F$, and for any $\lambda\in\lambda_{0}+\Lambda$,
\eq
\label{coh}
X_\lambda\otimes F = \sum_{\mu\in\Delta(F)} X_{\lambda+\mu},
\eeq
where $\Delta(F)$ denotes the multiset of all weights of $F$.
\end{enumerate}
See \cite{V2}, Definition 7.2.5. The reason that we may use coherent
families based on the compact Cartan $T$, rather than the maximally
split Cartan used in \cite{V2}, is our assumption that $G$ is
connected. %, for a more general setting; here we
% are only interested in coherent families based  
% on $\frt$. % Alternatively one can consider characters 
%
% \begin{equation*}
% \ch_\lambda\defn\ch(X_\lambda)
% \end{equation*}
%
% as functions on the regular elements of the compact Cartan subgroup
% $T$. Then the characterizing property (\ref{coh}) implies 
%
% \eq
% \label{cohchar}
% \ch_\lambda \ch(F) = \sum_{\mu\in\Delta(F)} \ch_{\lambda+\mu}.
% \eeq
%

A virtual $(\frg,K)$-module $X$ with regular infinitesimal character 
$\lambda_{0}\in\frh_{c}^{\star}$ can be placed in a unique coherent 
family as above (see Theorem 7.2.7 in \cite{V2}, and the references 
therein; this is equivalent to \eqref{existscoherent}).  Using this,
one can define an action of the integral Weyl  
group $W(\lambda_{0})$ attached to $\lambda_{0}$ on the set ${\mathcal
  M}(\lambda_{0})$ of virtual 
$(\frg,K)$-modules with infinitesimal character
$\lambda_{0}$. %DV {\clrr{
Recall that  
$W(\lambda_{0})$ consists of those elements $w\in W_\frg$ for which 
$\lambda_0-w\lambda_0$ is a sum of roots. If we write $Q$ for the root
lattice, then the condition for $w$ to be in $W(\lambda_0)$ is
precisely that $w$ preserves the lattice coset 
$\lambda_{0}+Q$ (see \cite{V2}, Section 7.2). Then for $w\in
    W(\lambda_0)$, we set 
\begin{equation*}
w\cdot X\defn X_{w^{-1}(\lambda_{0})}.
\end{equation*}
We view ${\mathcal M}(\lambda_{0})$ as a lattice (a free $\bZ$-module)
with basis the (finite) set of irreducible  
$(\fg,K)$-modules of infinitesimal character $\lambda_{0}$. %DV }}
A decomposition into irreducible components of this 
$W(\lambda_{0})$-representation, known as the {\it coherent continuation} 
representation, was obtained by Barbasch and Vogan (see \cite{BV1b}). 
The study of coherent continuation representations is important for 
deeper understanding of coherent families.  

% \begin{subequations}\label{se:cohirr} 
% Coherent families are closely related to Jantzen-Zuckerman translation
% functors, and they are important tools in the study of irreducible 
% $(\frg,K)$-modules. We fix now an irreducible $(\frg,K)$-module $X$
% having the (regular) infinitesimal character $\lambda_{0}$. 
A weight
$\lambda \in \lambda_0 + \Lambda$ is called {\it integrally dominant} if
\begin{equation}\label{intdom}
\langle\alpha^\vee,\lambda\rangle \ge 0 \ \text{whenever $\langle
  \alpha^\vee,\lambda_0 \rangle \in {\bbN}$} \qquad (\alpha \in R_\frg).
\end{equation} 
Recall from the introduction that we write $(\lambda_0 + \Lambda)^+$
for the cone of integrally dominant weights.
% A consequence of the relationship between translation functors and
% coherent families is
% \begin{equation}
% \text{$X_\lambda$ is irreducible or zero whenever $\lambda \in
%   (\lambda_0 + \Lambda)^+$;}
% \end{equation}
% zero is possible only for singular $\lambda$. (See for example
% \cite{V2}, sections 7.2 and 7.3.)
% \end{subequations} %se:cohirr

The notion of coherent families is closely related with
the Jantzen-Zuckerman translation principle.  
For example, if $\lambda$ is regular and $\lambda+\nu$ belongs to the
same Weyl chamber for integral roots (whose definition  
is recalled below), then $X_{\lambda+\nu}$  
can be obtained from $X_\lambda$ by a translation functor, i.e., by
tensoring with the finite-dimensional module $F_\nu$  
with extremal weight $\nu$ and then taking the component with
generalized infinitesimal character $\lambda+\nu$. 
The following observation is crucial for obtaining the translation
principle for Dirac index. 

\begin{prop}
\label{main}
Suppose $X$ is a virtual $(\frg,K)$-module and $F$ a finite-dimensional
$(\frg,K)$-module. Then
\[
I(X\otimes F)=I(X)\otimes F.
\]
\end{prop}
\pf By Proposition \ref{propindex} % DV {\clrr 
and (\ref{index formula}),
\[
I(X\otimes F)=X\otimes F\otimes S^+ - X\otimes F\otimes S^-,
\]
while
\[
I(X)\otimes F= (X\otimes S^+ - X\otimes S^-)\otimes F.
\]
It is clear that the right hand sides of these expressions are the same.
\epf
\noindent
Combining Proposition \ref{main} with (\ref{coh}), we obtain
\begin{cor}
\label{cohindex}
Let $X_\lambda$, $\lambda\in\lambda_{0}+\Lambda$, be a coherent family
of virtual $(\frg,K)$-modules and let $F$ be a  
finite-dimensional $(\frg,K)$-module. Then
\eq
\label{cohindexformula}
I(X_\lambda)\otimes F=\sum_{\mu\in\Delta(F)} I(X_{\lambda+\mu}).
\eeq
\qed
\end{cor}
\noindent
This says that the family
$\{I(X_\lambda)\}_{\lambda\in\lambda_{0}+\Lambda}$ of virtual
$\Kt$-modules has some  
coherence properties, but it is not a coherent family for $\Kt$, as
$I(X_\lambda)$ does not have $\frk$-infinitesimal character  
$\lambda$. Also, the identity (\ref{cohindexformula}) is valid 
only for a $(\frg,K)$-module $F$, and not for an arbitrary
$\widetilde{K}$-module $F$. 

Using standard reasoning, as in \cite{V2}, Section 7.2, we can now
analyze the relationship between Dirac index  
and translation functors. 
\begin{subequations}\label{Kcoherent}
%DV {\clrr{
We first define %DV }} 
some virtual representations of $\Kt$. Our choice of positive 
roots $R^+_\frk$ for $T$ in $K$ defines a Weyl denominator function
\begin{equation}\label{Weyldenominator}
d_\frk(\exp(y)) = \prod_{\alpha\in R^+_\frk} (e^{\alpha(y)/2} - e^{-\alpha(y)/2})
\end{equation}
on an appropriate cover of $T$. For $\gamma\in \Lambda+\rho_\frg$, the
Weyl numerator
\begin{equation*}
N_\gamma = \sum_{w\in W_\frk} \sgn(w) e^{w\gamma}
\end{equation*}
is a function on another double cover of $T$. According to Weyl's
character formula, the quotient
\begin{equation}
\ch_{\frk,\gamma} = N_\gamma/d_\frk
\end{equation}
extends to a class function on all of $\Kt$. Precisely,
$\ch_{\frk,\gamma}$ is the character of a virtual genuine representation
$\Et_\gamma$ of $\Kt$:
\begin{equation}
\Et_\gamma = \begin{cases} \sgn(x)\left(\text{irr. of highest
      weight $x\gamma - \rho_\frk$}\right)
&\text{$x\gamma$ is dom. reg. for $R^+_\frk$}\\
0 &\text{$\gamma$ is singular for $R_\frk$}
\end{cases}
\end{equation}
It is convenient to extend this definition to all of $\frt^*$ by
\begin{equation}
\Et_\lambda = 0 \qquad (\lambda \notin \Lambda+\rho_\frg).
\end{equation}

With this definition, the Huang-Pand\v zi\'c infinitesimal character
result clearly guarantees what we wrote in \eqref{indexformula}: % In
% other words, we may write    
%
\beu
I(X_{\lambda_{0}})=\sum_{w\in W_\frg} a_w \Et_{w\lambda_{0}}.
\end{equation*}
%
% where $\Et$ denotes the coherent family of finite-dimensional
% $\Kt$-modules and $a_w$ are integers. (
We could restrict the sum to
those $w$ for which $w\lambda_0$ is dominant for $R^+_\frk$, and get a
unique formula in which $a_w$ is the multiplicity of the $\Kt$
representation of highest weight $w\lambda_0 - \rho_\frk$ in
$I(X_{\lambda_0})$. But for the proof of the next theorem, it is more
convenient to allow a more general expression. % Then, for any
% $\nu\in\Lambda$, 
% From now on we will be assuming that $\lambda_{0}-\rho_\frk$
% exponentiates to%  the spin double cover $\widetilde{T}\subset\Kt$ of $T$.
% If this is not the case, then $I(X_\lambda)=0$ for all $\lambda$,
% and there is nothing further to say. We are going to make use of the
% coherent family  
% $\Et=\{\Et_\lambda\}_{\lambda\in\lambda_{0}+\Lambda}$ of
% finite-dimensional $\Kt$-modules, based on $\lambda_0+\Lambda$. See
% \cite{V2}, Example 7.2.12. (Note that the  
% lattice of $\Kt$-integral weights can be bigger than the lattice of
% $K$-integral weights, but in that case we restrict our family to  
% $\lambda_{0}+\Lambda$.)  
% Explicitly, $\Et_\lambda=0$ if $\lambda$ is singular for $\frk$, and if $\lambda$ is regular for $\frk$, 
% $\Et_\lambda=\det(u)\,E_{u\lambda-\rho_\frk}$, where $u\in W_\frk$
% is the element that conjugates $\lambda$ into the dominant chamber.
% Another way to think of the family $\Et_\lambda$ is in terms of the
% Weyl character formula for $\Kt$. 
\end{subequations} % Kcoherent

\begin{thm}
\label{translindex}
Suppose $\lambda_0\in\frt^*$ is regular for $\frg$. % and
% $\lambda_0-\rho_\frk$ exponentiates to $\widetilde{T}$.   
Let $X_\lambda$, $ \lambda\in\lambda_{0}+\Lambda$, be a coherent
family of virtual  
$(\frg,K)$-modules based on $\lambda_0 + \Lambda$. By Theorem \ref{HPmain}, we can write  
\eq
\label{indexatlambda}
I(X_{\lambda_{0}})=\sum_{w\in W_\frg} a_w \Et_{w\lambda_{0}},
\eeq
where $\Et$ denotes the family of finite-dimensional virtual
$\Kt$-modules defined in \eqref{Kcoherent}, and $a_w$ are integers. 
% (Note that there is a unique such expression for which $a_w=0$ unless
% $w\lambda_0$ is $\frk$-dominant. However it is more convenient for the
% proof to start with an arbitrary expression as above.) 

Then for any $\nu\in\Lambda$,
\eq
\label{indexatlambda+nu}
I(X_{\lambda_{0}+\nu})=\sum_{w\in W_\frg} a_w \Et_{w(\lambda_{0}+\nu)},
\eeq
with the same coefficients $a_w$.
\end{thm}
\pf
We proceed in three steps.

{\it Step 1:} suppose both $\lambda_{0}$ and $\lambda_{0}+\nu$ belong to the same integral Weyl chamber, which we can assume to be the dominant one. 
Let $F_\nu$ be the finite-dimensional $(\frg,K)$-module with extremal
weight $\nu$. Let us take the components of (\ref{cohindexformula}),
written for $\lambda=\lambda_0$, with $\frk$-infinitesimal characters which are
$W_\frg$-conjugate to $\lambda_{0}+\nu$. By Theorem \ref{HPmain}, any
summand $I(X_{\lambda_{0}+\mu})$ of the RHS of (\ref{cohindexformula})
is a combination of virtual modules with $\frk$-infinitesimal
characters which are $W_\frg$-conjugate to $\lambda_{0}+\mu$.  
By \cite{V2}, Lemma 7.2.18 (b), $\lambda_{0}+\mu$ can be
$W_\frg$-conjugate to $\lambda_{0}+\nu$ only if $\mu=\nu$. Thus we are
picking exactly the summand $I(X_{\lambda_{0}+\nu})$ of the RHS of
(\ref{cohindexformula}). 

We now determine the components of the LHS of (\ref{cohindexformula})
with $\frk$-infinitesimal characters which are  
$W_\frg$-conjugate to $\lambda_{0}+\nu$.
 Since $\Et$ is a coherent family for $\tilde{K}$, and $F_\nu$ can 
be viewed as a finite-dimensional $\tilde{K}$-module, one has 
\[
\Et_{w\lambda_{0}}\otimes F_\nu=\sum_{\mu\in\Delta(F_\nu)}\Et_{w\lambda_{0}+\mu}.
\]
The $\frk$-infinitesimal character of $\Et_{w\lambda_{0}+\mu}$ is $w\lambda_{0}+\mu$, so the components we are looking for must satisfy 
$w\lambda_{0}+\mu = u(\lambda_{0}+\nu)$, or equivalently
\[
\lambda_{0}+w^{-1}\mu=w^{-1}u(\lambda_{0}+\nu),
\]
for some $u\in W_\frg$. Using \cite{V2}, Lemma 7.2.18 (b) again, we see that $w^{-1}u$ must fix $\lambda_{0}+\nu$, and $w^{-1}\mu$ must be equal to $\nu$.
So $\mu=w\nu$, and the component $\Et_{w\lambda_{0}+\mu}$ is in fact $\Et_{w(\lambda_{0}+\nu)}$. So 
(\ref{indexatlambda+nu}) holds in this case.

 {\it Step 2:} suppose that $\lambda_{0}$ and $\lambda_{0}+\nu$ lie in two neighbouring chambers, with a common wall defined by a root $\alpha$, 
 and such that $\lambda_{0}+\nu=s_{\alpha}(\lambda_{0})$. Assume further that for any weight $\mu$ of $F_\nu$, $\lambda_{0}+\mu$ belongs to 
 one of the two chambers. Geometrically this means that $\lambda_{0}$ is close to the wall defined by $\alpha$ and sufficiently far from all other walls and from the origin. 
 We tensor (\ref{indexatlambda}) with $F_{\nu}$. By (\ref{cohindexformula}) and the fact that $\Et$ is a coherent family for $\tilde{K}$, we get
\begin{equation*}
\sum_{\mu\in\Delta(F_{\nu})}I(X_{\lambda_{0}+\mu})=\sum_{w\in W_{\frg}}a_{w}\sum_{\mu\in\Delta(F_{\nu})}\Et_{w(\lambda_{0}+\mu)}.
\end{equation*}
By our assumptions, the only $\lambda_{0}+\mu$ conjugate to $\lambda_{0}+\nu$ via $W_{\frg}$ are $\lambda_{0}+\nu$ and $\lambda_{0}$.
Picking the corresponding parts from the above equation, we get 
\begin{equation*}
I(X_{\lambda_{0}+\nu})+cI(X_{\lambda_{0}})=\sum_{w\in W_{\frg}}a_{w}\big(c\Et_{w\lambda_{0}}+\Et_{w(\lambda_{0}+\nu)}\big)
\end{equation*}
where $c$ is the multiplicity of the zero weight of $F_\nu$. This implies (\ref{indexatlambda+nu}), so the theorem 
is proved in this case.

{\it Step 3:} to get from an arbitrary regular $\lambda_{0}$ to an arbitrary $\lambda_{0}+\nu$, we first apply Step 1 to get from $\lambda_{0}$ to all 
elements of $\lambda_{0}+\Lambda$ in the same (closed) chamber. Then we apply Step 2 to pass to an element of a neighbouring chamber, then Step 1 again to get to all elements of that chamber, and so on.
\epf

\begin{cor}
\label{nonzeroindex}
In the setting of Theorem \ref{translindex}, assume that both $\lambda_{0}$ and $\lambda_{0}+\nu$ are regular for $\frg$.  
Assume also that $I(X_{\lambda_{0}})\neq 0$, i.e., at least one of the coefficients $a_w$ in (\ref{indexatlambda}) is nonzero. 
Then $I(X_{\lambda_{0}+\nu})\neq 0$.
\end{cor}
\pf
This follows immediately from Theorem \ref{translindex} and the fact that $\Et_{w(\lambda_{0}+\nu)}$ can not be zero, since
$w(\lambda_{0}+\nu)$ is regular for $\frg$ and hence also for $\frk$.
\epf

%%%%%%%%%%%%%%%%%%%%%%%%%%%%%%%%%%%%%%%%%%%%%%%%%%%%%%%%%%%%%%%%%%%%%%%%%%%%%%%%%%%%%%%
\section{Index polynomial and coherent continuation representation}
\label{section Weyl group}

As in the previous section, let $\lambda_{0}\in\frt^{\star}$ be regular. For each $X\in {\mathcal M}(\lambda_{0})$, there is a unique coherent family 
$\{X_{\lambda}\mid \lambda\in\lambda_{0}+\Lambda\}$ such that $X_{\lambda_{0}}=X$. Define a function $Q_X\colon \lambda_{0}+\Lambda \to\bbZ$ by setting
\eq
\label{dim}
Q_{X}(\lambda)= \dim I(X_\lambda)\quad (\lambda\in\lambda_{0}+\Lambda).
\eeq
Notice that $Q_{X}$ depends on both $X$ {\it and} on the choice of
representative $\lambda_{0}$ for the infinitesimal character of $X$;
replacing $\lambda_{0}$ by  
$w_{1}\lambda_{0}$ translates $Q_{X}$ by $w_{1}$.
By Theorem \ref{translindex} and the Weyl dimension formula for
$\frk$, $Q_{X}$ is a polynomial function in $\lambda$.  
(Note that taking dimension is additive with respect to short exact
sequences of finite-dimensional  
modules, so it makes sense for virtual finite-dimensional modules.) We
call the function $Q_{X}$ the  
{\it index polynomial} associated with $X$ (or $\{X_{\lambda}\}$). 

Recall that a polynomial on $\frt^*$ is called $W_\frg$-harmonic, if
it is annihilated by any $W_\frg$-invariant constant coefficient 
differential operator on $\frt^*$ without constant term (see
\cite{V1}, Lemma 4.3.) 

\begin{prop}
\label{harmonic}
For any $(\frg,K)$-module $X$ as above, the index polynomial $Q_X$ is
$W_\frg$-harmonic. If $Q_X\neq 0$, then it is homogeneous of degree
equal to the number of positive roots for $\frk$; more precisely, it
belongs to the irreducible representation of $W_{\frg}$ generated by
the Weyl dimension formula for $\frk$. 
\end{prop} 
\pf The last statement follows from $(\ref{indexatlambda+nu})$; the rest of the proposition is an immediate consequence.
\epf

Recall the natural representation of $W(\lambda_{0})$ (or indeed of all $W_{\frg}$) on the vector space $S(\frt)$ of polynomial functions on $\frt^*$, 
$$
(w\cdot P)(\lambda)=P(w^{-1}\lambda).
$$
The (irreducible) representation of $W(\lambda_{0})$ generated by the dimension formula for $\frk$ is called the {\it index polynomial representation}.
\begin{prop} 
\label{wequi}
The map
\begin{equation*}
{\mathcal M}(\lambda_{0})\rightarrow S(\frt),\qquad X\mapsto Q_{X}
\end{equation*}
intertwines the coherent continuation representation of $W(\lambda_0)$ with the action on polynomials. In particular, if $Q_{X}\neq 0$, then the coherent continuation representation 
generated by $X$ must contain a copy of the index polynomial representation.
\end{prop}
\pf
 Let $\{X_\lambda\}$ be the coherent family corresponding to $X$. Then for a fixed $w\in W(\lambda_0)$, the coherent 
 family corresponding to $w\cdot X$ is $\lambda_{0}+\nu\mapsto X_{w^{-1}(\lambda_{0}+\nu)}$ (see \cite{V2}, Lemma 7.2.29 and its proof). It follows that
\begin{eqnarray*}
(w\cdot Q_{X})(\lambda)&=&Q_{X}(w^{-1}\cdot\lambda)\\
&=&\dim I(X_{w^{-1}\lambda})\\
&=&Q_{w\cdot X}(\lambda),
\end{eqnarray*}
i.e., the map $X\mapsto Q_{X}$ is $W(\lambda_{0})$-equivariant. The rest of the proposition is now clear.
\epf
\begin{ex}
\label{exfd}
{\rm Let $F$ be a finite-dimensional $(\frg,K)$-module. The 
  corresponding coherent family is  
$\{F_\lambda\}$ from \cite{V2}, Example 7.2.12. In particular, every $F_\lambda$ is finite-dimensional up to sign, 
or 0. By Proposition \ref{propindex} %DV {\clrr 
and (\ref{index formula}), for any $F_\lambda$, 
\[
\dim I(F_\lambda)=\dim(F_\lambda\otimes S^+-F_\lambda\otimes S^-)=\dim
F_\lambda(\dim S^+-\dim S^-)=0, 
\]
since $S^+$ and $S^-$ have the same dimension (as long as $\frp\neq
0$). It follows that  
\[
Q_{F}(\lambda)=0.
\]
(Note that the index itself is a nonzero virtual module, but its dimension is zero. This may be a little surprising at first, but it is quite possible for virtual modules.)
This means that in this case Proposition \ref{wequi} gives no information about the coherent continuation representation (which is in this case a copy of the sign representation of $W_\frg$ spanned by $F$).}
\end{ex}

\begin{ex} 
\label{exds_sl2} 
{\rm Let $G=SL(2,\bbR)$, so that weights correspond to integers. Let $\lambda_{0}=n_0$ be a positive integer.
There are four irreducible $(\frg,K)$-modules with infinitesimal character $n_0$: the finite-dimensional module $F$,
the holomorphic discrete series $D^+$ of lowest weight $n_0+1$, the antiholomorphic discrete series $D^-$ of highest weight $-n_0-1$, and the irreducible principal series representation $P$. 

The coherent family $F_n$ corresponding to $F$ is defined by setting $F_n$ to be the finite-dimensional module with highest weight $n-1$ if $n>0$, $F_0=0$, and if $n<0$, $F_n=-F_{-n}$. Thus $s\cdot F=-F$, i.e., $F$ spans a copy of the sign representation of $W(\lambda_{0})=\{1,s\}$. 
As we have seen, the index polynomial corresponding to $F$ is zero. 

By \cite{V2}, Example 7.2.13, the coherent family $D_n^+$ corresponding to $D^+$ is given as follows: for $n\geq 0$, 
$D^+_n$ is the irreducible lowest weight $(\frg,K)$-module with lowest weight $n+1$, and for 
$n<0$, $D^+_n$ is the sum of $D^+_{-n}$ and the finite-dimensional module $F_{-n}$. It is easy to see that for each 
$n\in\bbZ$, $I(D^+_n)$ is the one-dimensional $\widetilde K$-module $E_n$ with weight $n$. So the index polynomial 
$Q_{D^+}$ is the constant polynomial $1$. Moreover, $s\cdot D^+=D^++F$.

One similarly checks that the coherent family $D_n^-$ corresponding to $D^-$ is given as follows: for $n\geq 0$, 
$D^-_n$ is the irreducible highest weight $(\frg,K)$-module with highest weight $-n-1$, and for 
$n<0$, $D^-_n=D^-_{-n}+F_{-n}$. For each $n\in\bbZ$, $I(D^-_n)=-E_{-n}$, so the index polynomial 
$Q_{D^-}$ is the constant polynomial $-1$. Moreover, $s\cdot D^-=D^-+F$.

Finally, one checks that the coherent family corresponding to $P$ consists entirely of principal series representations, that the $W(\lambda_{0})$-action on $P$ is trivial, and that the corresponding index polynomial is 0. 

Putting all this together, we see that the coherent continuation representation at $n_0$ consists of three trivial representations, spanned by $F+D^++D^-$, $D^+-D^-$ and $P$, and one sign representation, spanned by $F$. The index polynomial representation is the trivial representation spanned by the constant polynomials. The map $X\mapsto Q_X$ 
sends $P$, $F$ and $F+D^++D^-$ to zero, and $D^+-D^-$ to the constant polynomial $2$.
}
\end{ex}

The conclusion of Example \ref{exfd} about the index polynomials of finite-dimensional representations being zero can be generalized as follows.

\begin{prop}
\label{indexzero} 
Let $X$ be a $(\frg,K)$-module as above, with Gelfand-Kirillov dimension $\Dim(X)$. 
If $\Dim(X)<\sharp R_{\frg}^{+}-\sharp R_{\frk}^{+}$, then $Q_X=0$.
\end{prop}

\pf We need to recall the setting of \cite{BV2}, Section 2, in particular their Theorem 2.6.(b) (taken from \cite{J1II}). Namely, to any 
irreducible representation $\sigma$ of $W_\frg$ one can associate its degree, i.e., the minimal integer $d$ such that $\sigma$ occurs in the $W_\frg$-representation $S^d(\frt)$. Theorem 2.6.(b) of \cite{BV2} says that the degree of any $\sigma$ occurring in the coherent continuation representation attached to $X$ must be at least equal to $\sharp R_{\frg}^{+}-\Dim(X)$. By assumption, the degree of $Q_X$, $\sharp R_{\frk}^{+} $, is smaller than 
$\sharp R_{\frg}^{+}-\Dim(X)$. On the other hand, by Proposition \ref{wequi} the index polynomial representation has to occur in the coherent continuation representation. It follows that $Q_X$ must be zero.
\epf

\begin{ex}
{\rm
Wallach modules for $Sp(2n,\bbR)$, $SO^*(2n)$ and $U(p,q)$, studied in \cite{HPP}, all have nonzero index, but their index polynomials are zero. This can also be checked explicitly from the results of \cite{HPP}, at least in low-dimensional cases. 

The situation here is like in Example \ref{exfd}; the nonzero Dirac index has zero dimension. In particular the conclusion $Q_X=0$ in Proposition \ref{indexzero} does not imply that $I(X)=0$.
}
\end{ex}

We note that in the proof of Proposition \ref{indexzero}, we are
applying the results of \cite{BV2} to $(\frg,K)$-modules, although
they are stated in \cite{BV2}  
for highest weight modules. This is indeed possible by results of
Casian \cite{C}. We explain this in more detail. Let ${\mathcal B}$ be
the flag variety of $\fg$ consisting of all the Borel subalgebras of
$\fg$. For a point $x\in{\mathcal B}$, write
$\frb_{x}=\frh_{x}+\frn_{x}$ for the corresponding  
Borel subalgebra, with nilradical $\frn_{x}$, and Cartan subalgebra
$\frh_{x}$. % \textcolor{red}{obtained from the abstract Cartan algebra
% via specialization}.  

Define a functor $\Gamma_{\frb_{x}}$ from the category of
$\frg$-modules into the category of $\frg$-modules which are
$\frb_x$-locally finite, by 
\beu
\Gamma_{\frb_{x}}M=\big\{\frb_{x}-\text{locally finite vectors in } M\big\}.
\end{equation*}
Write $\Gamma^{q}_{\frb_{x}}$, $q\geq 0$, for its right derived functors. Instead of considering the various $\frb_{x}$, $x\in{\mathcal B}$, it is convenient to fix a Borel 
subalgebra $\frb=\frh+\frn$ of $\frg$ and twist the module $M$. By a twist of $M$ we mean that if $\pi$ is the $\frg$-action on $M$ and $\sigma$ is an automorphism of 
$\frg$ then the twist of $\pi$ by $\sigma$ is the $\frg$-action $\pi\circ\sigma$ on $M$. Then Casian's generalized Jacquet functors  $J_{\frb_{x}}^{q}$ are functors from the category of $\frg$-modules into the category of $\frg$-modules which are $\frb$-locally finite, given by 
\beu
J_{\frb_{x}}^{q}M=\Big\{\Gamma_{\frb_{x}}^{q}\Hom_{\bC}(M,\bC)\Big\}^{0}
\end{equation*}
where the superscript `0' means that the $\frg$-action is twisted by some inner automorphism of $\frg$, to make it $\frb$-locally finite instead of $\frb_{x}$-locally finite. 
In case $\frb_{x}$ is the Borel subalgebra corresponding to an Iwasawa decomposition of $G$,  
 $J^0_{\frb_x}$ is the usual Jacquet functor of \cite{BB}, while the $J_{\frb_{x}}^q$ vanish for $q>0$. 

The functors $J_{\frb_{x}}^{q}$ make sense on the level of virtual $(\frg,K)$-modules and induce an injective map 
\beu
X\mapsto \sum_{x\in{\mathcal B}/K}\sum_{q}(-1)^{q}J_{\frb_{x}}^{q}X
\end{equation*}
from virtual $(\frg,K)$-modules into virtual $\frg$-modules which are $\frb$-locally finite. Note that the above sum is well defined, since the 
$J_{\frb_{x}}^{q}$ depend only on the $K$-orbit of $x$ in $\mathcal{B}$.

An important feature of the functors $J_{\frb_{x}}^{q}$ is the fact that 
they satisfy the following identity relating the $\frn_{x}$-homology of $X$ with the $\frn$-cohomology of the modules $J_{\frb_{x}}^{q}X$ (see page 6 in \cite{C}): 
\beu
\sum_{p,q\geq 0}(-1)^{p+q}\tr_\frh H^{p}(\frn,J_{\frb_{x}}^{q}X)=\sum_{q}(-1)^{q}\tr_\frh H_{q}(\frn_{x},X)^{0}.
\end{equation*}
Here the superscript `0' is the appropriate twist interchanging $\frh_{x}$ with $\frh$, and 
 $\tr_\frh$ denotes the formal trace of the $\frh$-action. More precisely, if $Z$ is a locally finite $\frh$-module with finite-dimensional weight
 components $Z_\mu$, $\mu\in\frh^*$, then 
 \[
 \tr_\frh Z=\sum_{\mu\in\frh^*} \dim Z_\mu\, e^\mu.
 \]
 
Using this and Osborne's character formula, the global character of $X$ on an arbitrary $\theta$-stable Cartan subgroup can be read off from the characters of the $J^q_{\frb_{x}}X$ (see \cite{C} and 
\cite{C2}). In particular, we deduce that if $\tau$ is an irreducible representation of the Weyl group $W_{\frg}$ occuring in the coherent representation attached to $X$ 
then $\tau$ occurs in the coherent continuation representation attached to $J_{\frb_{x}}^{q}X$ for some $q\geq 0$ and some Borel subalgebra $\frb_{x}$. Moreover, 
from the definitions, one has $\Dim(X)\geq \Dim(J_{\frb_{x}}^{q}X)$. Applying the results in \cite{BV2} to the module $J_{\frb_{x}}^{q}X$, we deduce that:
\beu
d^{o}(\tau)\geq \sharp R^{+}_{\frg}-\Dim(J_{\frb_{x}}^{q}X)\geq  \sharp R^{+}_{\frg}-\Dim(X),
\end{equation*}
where $d^{o}(\tau)$ is the degree of $\tau$.

%%%%%%%%%%%%%%%%%%%%%%%%%%%%%%%%%%%%%%%%%%%%%%%%%%%%%%%%%%%%%%%%%%%%%%%%%%%%%%%%%%%%%%%%%%
\section{Index polynomials and Goldie rank polynomials} \label{section Goldie rank}

Recall that $H_s$ denotes a maximally split Cartan subgroup of $G$ with complexified Lie algebra $\frh_s$.
As in Section \ref{section coherent}, we let $X$ be a module with regular infinitesimal character $\lambda_{0}\in\frh_s^{\star}$, and 
$\{X_{\lambda}\}_{\lambda\in\lambda_0+\Lambda}$ the corresponding coherent family on $H_s$. 
With notation from (\ref{annintro}) and (\ref{goldieintro}), Joseph proved that the mapping 
\eqn
\lambda\mapsto P_{X}(\lambda)=\rk (U(\frg)/\Ann(X_{\lambda})),
\eeqn
extends to a $W_\frg$-harmonic polynomial on $\frh_s^*$, homogeneous
of degree $\sharp R^{+}_{\frg}-\Dim(X)$, where $\Dim(X)$ is the
Gelfand-Kirillov dimension of $X$ (see \cite{J1I}, \cite{J1II} and
\cite{J1III}). He also  
found relations between the Goldie rank polynomial %DV {\clrr 
$P_X$ and Springer
representations; and (less directly) Kazhdan-Lusztig polynomials (see
\cite{J2} and \cite{J3}).  

Recall from \eqref{kingintro} King's analytic interpretation of the
Goldie rank polynomial: that for $x\in \frh_{s,0}$ regular, the expression
% In \cite{K1} and \cite{K2} King found an interpretation of the
% Goldie rank polynomials attached to Harish-Chandra modules in terms
% of the asymptotics of the global characters $\ch_\frg$ on the
% maximally split Cartan subgroup. He proved that the expression 
\eq
\label{charpol}
 \lim_{t\to 0+} t^d\ch_\frg(X_\lambda)(\exp tx)
\eeq
is zero if $d$ is an integer bigger than $\Dim(X)$; and if
$d=\Dim(X)$, it is (for generic $x$) a nonzero polynomial $C_{X,x}$ in
$\lambda$ called the character polynomial. Up to a constant, this
character polynomial is equal to the Goldie rank polynomial attached
to $X$. In other words, the Goldie rank polynomial expresses 
the dependence on $\lambda$ of the leading term in the Taylor
expansion of the numerator of the character of $X_\lambda$ on the
maximally split Cartan $H_{s}$.
%  is up
% to a constant equal to the leading term in the Taylor expression of
% the character on the maximally split Cartan subgroup. 
 For more details, see \cite{K1} and also \cite{J1II}, Corollary 3.6.

The next theorem shows that the index polynomial we studied in Section
\ref{section Weyl group} is the exact analogue of King's character
polynomial, but attached to the character on the compact Cartan
subgroup instead of the maximally split Cartan subgroup.  

\begin{thm}
\label{ind=char} 
Let $X$ be a $(\frg,K)$-module with regular infinitesimal character
and let $X_\lambda$ be the corresponding coherent family on the
compact Cartan subgroup. Write  
$r_\frg$ (resp. $r_\frk$) for the number of positive $\frt$-roots for
$\frg$ (resp. $\frk$). Suppose $y\in \frt_0$ is any regular
element. Then the limit 
\eq
\label{indexpol}
\lim_{t\to 0+} t^d \ch_\frg(X_\lambda)(\exp ty)
\eeq
is zero if $d$ is an integer bigger than $r_\frg-r_\frk$. If
$d=r_\frg-r_\frk$, then the limit \eqref{indexpol} is equal to  
\[
\textstyle{\frac{\prod_{\alpha\in R_\frk^+}\alpha(y)}{\prod_{\alpha\in
      R_\frg^+}\alpha(y)}}\, Q_X(\lambda), 
\] 
where $Q_X$ is the index polynomial attached to $X$ as in
(\ref{dim}). In other words, the index polynomial, up to an explicit
constant, expresses  
the dependence on $\lambda$ of the (possibly zero) leading term in the Taylor
expansion of the numerator of the character of $X_\lambda$ on the
compact Cartan $T$.   
\end{thm}
\pf
The restriction to $K$ of any $G$-representation has a well defined
distribution character, known as the $K$-character. The restriction of
this $K$-character  
to the set of elliptic $G$-regular elements in $K$ is a function,
equal to the function giving the $G$-character % DV {\clrr 
(see \cite{HC}, and also \cite{AS}, (4.4) 
and the appendix). Therefore Proposition \ref{propindex} %DV {\clrr 
and (\ref{index formula}) imply
\eqn
\ch_\frg(X_\lambda)(\exp
ty)=\frac{\ch_\frk(I(X_\lambda))}{\ch_\frk(S^+-S^-)}(\exp ty). 
\eeqn
Also, it is clear that
\[
\lim_{t\to 0+} \ch_\frk(I(X_\lambda))(\exp
ty)=\ch_\frk(I(X_\lambda))(e)=\dim I(X_\lambda)=Q_X(\lambda). 
\]
Therefore the limit (\ref{indexpol}) is equal to 
\[
\lim_{t\to 0+} t^d \frac{\ch_\frk(I(X_\lambda))(\exp
  ty)}{\ch_\frk(S^+-S^-)(\exp ty)}=Q_X(\lambda)\lim_{t\to 0+}   
\frac{t^d}{\ch_\frk(S^+-S^-)(\exp ty)}.
\]
On the other hand, it is well known and easy to check that
\[
\ch_\frk(S^+-S^-)=\frac{d_\frg}{d_\frk},
\]
where $d_\frg$ (resp. $d_\frk$) denotes the Weyl denominator for
$\frg$ (resp. $\frk$). It is immediate from the product formula
\eqref{Weyldenominator} that  %From the proof of the Weyl dimension formula,
we know that 
\[
d_\frg(\exp ty)=t^{r_\frg}\prod_{\alpha\in R_\frg^+}
\alpha(y) + \text{ higher order terms in } t 
\]
and similarly
\[
d_\frk(\exp ty)=t^{r_\frk}\prod_{\alpha\in R_\frk^+}\alpha(y) + \text{ higher order terms in } t. 
\]
So we see that
\[
\lim_{t\to 0+} \frac{t^d}{\ch_\frk(S^+-S^-)(\exp ty)}=\lim_{t\to 0+}
t^{d-r_\frg+r_\frk} \frac{\prod_{\alpha\in
    R_\frk^+}\alpha(y)}{\prod_{\alpha\in R_\frg^+} \alpha(y)}. 
\]
The theorem follows.

\epf

We are now going to consider some examples (of discrete series
representations) where we compare the index polynomial and the Goldie
rank polynomial. To do so, we identify the compact Cartan subalgebra
with the maximally split one using a Cayley transform. 

Recall that if $X$ is a discrete series representation with
Harish-Chandra parameter $\lambda$, then 
\[
I(X)= \pm H_D(X)= \pm E_\lambda,
\]
where $E_\lambda$ denotes the $\Kt$-type with infinitesimal character
$\lambda$. (The sign depends on the relation between the positive
system defined by $\lambda$ and the fixed one used in Section
\ref{section index} to define the index. See \cite{HP1}, Proposition
5.4, or \cite{HP2}, Corollary 
7.4.5.)  The index polynomial $Q_X$ is then given by the Weyl dimension formula
for this $\Kt$-type, i.e., by  
\eq 
\label{indexds}
Q_X(\lambda)=\prod_{\alpha\in R_\frk^+}
\frac{\langle\lambda,\alpha\rangle}{\langle\rho_\frk,\alpha\rangle}. 
\eeq
Comparing this with \cite{K2}, Proposition 3.1, we get:

\begin{prop}
\label{holods} 
Suppose $G$ is linear, semisimple and of Hermitian type. Let $X$ be
the $(\frg,K)$-module of a holomorphic discrete series
representation. Then the index polynomial $Q_X$ coincides with the
Goldie rank polynomial $P_X$ up to a scalar multiple. 
%\qed
\end{prop} 
Of course, $Q_X$ is not always equal to $P_X$, since the degrees of
these two polynomials are different in most cases.  

In the following we consider the example of discrete series
representations for $SU(n,1)$. The choice is dictated by the existence
of explicit formulas for the Goldie rank polynomials computed in
\cite{K2}. 

The discrete series representations for $SU(n,1)$ with a fixed
infinitesimal character can be parametrized by integers $i\in
[0,n]$. To see how this works, we introduce some notation. First, we
take for $K$ the group $S(U(n)\times U(1))\cong U(n)$. The compact
Cartan subalgebra $\frt$ consists of diagonal matrices, and we
identify it with $\bbC^{n+1}$ in the usual way. We make the usual
choice for the dominant $\frk$-chamber $C$: it consists of those
$\lambda\in\bbC^{n+1}$ for which 
\[
 \lambda_1\geq\lambda_2\geq\dots \geq\lambda_n.
\]
Then $C$ is the union of $n+1$ $\frg$-chambers $D_0,\dots,D_n$, where $D_0$ consists of $\lambda\in C$ such that $\lambda_{n+1}\leq \lambda_n$, 
$D_n$ consists of $\lambda\in C$ such that $\lambda_{n+1}\geq \lambda_1$, and for $1\leq i\leq n-1$,
\[
D_i=\{\lambda\in C\,\big|\, \lambda_{n-i}\geq \lambda_{n+1}\geq\lambda_{n-i+1}\}.
\]
Now for $i\in [0,n]$, and for $\lambda\in D_i$, which is regular for
$\frg$ and analytically integral for $K$, we denote by $X_\lambda(i)$
the discrete series 
representation with Harish-Chandra parameter $\lambda$. We use the 
same notation for the corresponding $(\frg,K)$-module. For $i=0$,
$X_\lambda(i)$ is holomorphic and this case is settled by Proposition
\ref{holods}; the result is that both the index polynomial and the
Goldie rank polynomial are proportional to the Vandermonde determinant

\eq
\label{vandermonde}
V(\lambda_1,\dots,\lambda_n)=\prod_{1\leq p<q\leq n}(\lambda_p-\lambda_q).
\eeq
The case $i=n$ of antiholomorphic discrete series representations is
analogous. For  
$1\leq i\leq n-1$, the index polynomial of $X_\lambda(i)$ is still
given by (\ref{vandermonde}). On the other hand, the character
polynomial is up to a constant multiple given by the formula (6.5) of
\cite{K2}, as the sum of two determinants. We note that King's
expression can be simplified and that the character polynomial of
$X_\lambda(i)$ is in fact equal to 

\eq
\label{chards}
\left|\begin{matrix} \lambda_1^{n-2}&\dots&\lambda_{n-i}^{n-2}&\lambda_{n-i+1}^{n-2}&\dots&\lambda_n^{n-2} \cr
      \lambda_1^{n-3}&\dots&\lambda_{n-i}^{n-3}&\lambda_{n-i+1}^{n-3}&\dots&\lambda_n^{n-3} \cr
\vdots&&\vdots&\vdots&&\vdots \cr
\lambda_1&\dots&\lambda_{n-i}&\lambda_{n-i+1}&\dots&\lambda_n \cr
1&\dots&1&0&\dots&0 \cr
0&\dots&0&1&\dots&1 \end{matrix}
\right|
\eeq
\smallskip

\noindent up to a constant multiple. 

For $i=1$, (\ref{chards}) reduces to the Vandermonde determinant $V(\lambda_1,\dots,\lambda_{n-1})$. Similarly, for $i=n-1$, we get
$V(\lambda_2,\dots,\lambda_n)$. In these cases, the Goldie rank polynomial divides the index polynomial.

For $2\leq i\leq n-2$, the Goldie rank polynomial is more complicated. For example, if $n=4$ and $i=2$, (\ref{chards}) becomes 
\[
-(\lambda_1-\lambda_2)(\lambda_3-\lambda_4)(\lambda_1+\lambda_2-\lambda_3-\lambda_4),
\]
and this does not divide the index polynomial. For $n=5$ and $i=2$, (\ref{chards}) becomes
\begin{multline*}
-(\lambda_1-\lambda_2)(\lambda_1-\lambda_3)(\lambda_2-\lambda_3)(\lambda_4-\lambda_5) \\
(\lambda_1\lambda_2+\lambda_1\lambda_3-\lambda_1\lambda_4-\lambda_1\lambda_5+\lambda_2\lambda_3
-\lambda_2\lambda_4-\lambda_2\lambda_5-\lambda_3\lambda_4-\lambda_3\lambda_5+\lambda_4^2+\lambda_4\lambda_5+\lambda_5^2),
\end{multline*}
and one can check that the quadratic factor is irreducible.

More generally, for any $n\geq 4$ and $2\leq i\leq n-2$, the Goldie rank polynomial (\ref{chards}) is divisible by $(\lambda_p-\lambda_q)$ whenever
$1\leq p<q\leq n-i$ or $n-i+1\leq p<q\leq n$. This is proved by subtracting the $q$th column from the $p$th column. On the other hand, if
$1\leq p\leq n-i<q\leq n$, we claim that (\ref{chards}) is not divisible by $(\lambda_p-\lambda_q)$. Indeed, we can substitute $\lambda_q=\lambda_p$ into (\ref{chards}) and subtract the $q$th column from the $p$th column. After this we develop the determinant with respect to the $p$th column. The resulting sum of two determinants is equal to the Vandermonde determinant $V(\lambda_1,\dots,\lambda_{p-1},\lambda_{p+1},\dots,\lambda_n)$, and this is not identically zero.

This proves that for $X=X_\lambda(i)$ the greatest common divisor of $P_X$ and $Q_X$ is 
\eq
\label{gcd}
\prod_{1\leq p<q\leq n-i}(\lambda_p-\lambda_q)\prod_{n-i+1\leq r<s\leq n}(\lambda_r-\lambda_s).
\eeq
Comparing with the simple roots $\Psi_i$ corresponding to the chamber $D_i$ described on p. 294 of \cite{K2}, we see that the linear factors of (\ref{gcd})
correspond to roots generated by the compact part of $\Psi_i$. On the other hand, the set of compact roots in $\Psi_i$ is equal to the $\tau$-invariant of $X_\lambda(i)$, as proved in \cite{HS}, Proposition 3.6 (see also \cite{K1}, Remark 4.5). Recall that 
the $\tau$-invariant of a $(\frg,K)$-module $X$ consists of the simple roots $\alpha$ such that the translate of $X$ to the wall defined by $\alpha$ is 0; see \cite{V1}, Section 4.

In particular, we have checked a special case of the following proposition.

\begin{prop}
\label{tau}
Assume that $G$ is a real reductive Lie group in the Harish-Chandra class and that $G$ and $K$ have equal rank. Let $X$ be the discrete series representation of $G$ with Harish-Chandra parameter $\lambda$.
Then the index polynomial $Q_X$ and the Goldie rank polynomial $P_X$ are both divisible by the product of linear factors 
corresponding to the roots generated by the $\tau$-invariant of $X$.
\end{prop}

\pf The $\tau$-invariant of $X$ is still given as above, as the compact part of the simple roots corresponding to $\lambda$. In particular, the roots generated by the $\tau$-invariant are all compact, and the corresponding factors divide $Q_X$, which is given by (\ref{indexds}). 

On the other hand, by \cite{V1}, Proposition 4.9, the Goldie rank polynomial is always divisible by the factors corresponding to roots generated by the $\tau$-invariant. We note that the result in \cite{V1} is about the Bernstein degree polynomial, which is up to a constant factor equal to the Goldie rank polynomial
by \cite{J1II}, Theorem 5.7.
\epf

Note that for $G=SU(n,1)$, the result we obtained is stronger than the
conclusion of Proposition \ref{tau}. Namely, we proved that the
product of linear factors  
corresponding to the roots generated by the $\tau$-invariant of $X$ is
in fact the greatest common divisor $R$ of $P_X$ and $Q_X$. We note
that it is easy to calculate the degrees of all the polynomials
involved. Namely, if $2\leq i\leq n-2$, the degree of $R$ is
$\binom{i}{2}+\binom{n-i}{2}$. Since $\Dim(X)=2n-1$ (see \cite{K2}),
and $\sharp R_{\frg}^{+}=\binom{n+1}{2}$, 
the degree of $P_X$ is $\binom{n-1}{2}$. It follows that the degree of
$P_X/R$ is $i(n-i)-(n-1)$. On the other hand, since the degree of 
$Q_X$ is $\sharp R_{\frk}^{+}=\binom{n}{2}$, the degree of $Q_X/R$ is
$i(n-i)$. 
%%%%%%%%%%%%%%%%%%%%%%%%%%%%%%%%%%%%%%%%%%%%%%%%%%%%%%%%%%%%%%%%%%%%%%%%%%%%%%%%%%%%%

\section{Index polynomials and nilpotent orbits} \label{orbits}

\begin{subequations}\label{Korbit}
Assume again that we are in the setting \eqref{se:cohintro} of the
introduction, so that $Y=Y_{\lambda_0}$ is an irreducible
$(\frg,K)$-module. (We use a different letter from the $X$ in the
introduction as a reminder that we
will soon be imposing some much stronger additional hypotheses on
$Y$.) Recall from \eqref{multintro} the expression
\begin{equation}% \label{eq:multintro}
\Ass(Y_\lambda) = \coprod_{j=1}^r m^j_Y(\lambda) \overline{{\mathcal
    O}^j} \qquad (\lambda \in (\lambda_0 + \Lambda)^+), % \ \text{integrally
%  dominant}),
\end{equation}
and the fact that each $m^j_Y$ extends to a polynomial function on
$\frt^*$, which is a multiple of the Goldie rank polynomial:
\begin{equation}
m^j_Y = a^j_Y P_Y,
\end{equation}
with $a^j_Y$ a nonnegative rational number depending on $Y$.
On the other hand, the Weyl dimension formula for $\frk$ defines a
polynomial on the dual of the compact Cartan subalgebra $\frt^*$ in
$\frg$, with 
degree equal to the cardinality $\sharp R_{\frk}^{+}$ of positive
roots for $\frk$. Write $\sigma_{K}$ for the representation of the
Weyl group $W_{\frg}$ generated by this polynomial. Suppose that
$\sigma_{K}$ is a Springer representation, i.e., it is associated with a
nilpotent $G_{\bC}$-orbit ${\mathcal O}_{K}$: % in $\frg^{\star}$:
\begin{equation}\label{assum1}
\sigma_K \overset{\text{Springer}}\longleftrightarrow {\mathcal O}_K
  \subset \frg^*.  
\end{equation}

%. { \clrr
  Here $G_\bbC$ denotes a connected complex reductive algebraic 
group having Lie algebra $\frg$. % $G$ as a real form.}  
Assume also that there is a Harish-Chandra module $Y$ %DV 
of regular infinitesimal character $\lambda_0$ % DV
such that
\begin{equation}\label{assum2}
{\mathcal V}(\gr(\Ann(Y)))=\overline{{\mathcal O}_{K}}.
\end{equation}
% where $\overline{{\mathcal O}_{K}}$ denotes the Zariski closure of $O_{K}$. 
% DV {\clrr{
Recall from the discussion before (\ref{eq:Korbit}) that ${\mathcal
  V}(\gr(\Ann(Y)))$ is the variety associated with the graded ideal  
of $\Ann(Y)$ in the symmetric algebra $S(\frg)$. %DV }}

% Next, since one has ${\mathcal V}(Y)\subseteq{\mathcal
%  V}(\gr(\Ann(Y)))$ and ${\mathcal V}(Y)\subseteq
% (\frg/\frk)^{\star}$, we see that 
%
% \beu 
% {\mathcal V}(Y)\subseteq\overline{{\mathcal O}_{K}}\cap(\frg/\frk)^{\star}.
% \end{equation*} 
%
% We can decompose $\overline{{\mathcal O}_{K}}\cap(\frg/\frk)^{\star}$
% into a union of finitely many Zariski closures of nilpotent
% $K_{\bC}$-orbits ${\mathcal O}_{K}^{j}$ in $(\frg/\frk)^{\star}$:
% \[
% \overline{{\mathcal O}_{K}}\cap(\frg/\frk)^{\star}=\bigcup_{j}
% m^j\;\overline{{\mathcal O}_{K}^{j}}, 
% \]
% where $m^j$ denotes the multiplicity of $\overline{{\mathcal
%     O}_{K}^{j}}$. 
% It follows that the variety ${\mathcal V}(Y)$ splits into  
%
% \beu 
% {\mathcal V}(Y)=\bigcup_{j}m^j_{Y}\;\overline{{\mathcal O}_{K}^{j}} 
% \end{equation*} 
%
% with $0\leq m^j_Y\leq m^j$.  
% The multiplicities $m^j_Y$ depend on the module  
% $Y$ and coincide with the multiplicities of the irreducible components
% occurring in the associated cycle of $Y$ (see \cite{V3}). It is known
% that half the codimension of ${\mathcal O}_{K}$  
% in the cone $\mathcal{N}$ of nilpotent elements in $\frg^{\star}$
% equals $\sharp R_{\frk}^{+}$ (see \cite{SV}), while half the dimension
% of ${\mathcal O}_{K}$ equals the Gelfand-Kirillov dimension of $Y$  
% (see \cite{V1}). Since $\dim\mathcal{N}=2\, \sharp R_{\frg}^{+}$, it
% follows that 

Our assumptions force the degree of the Goldie rank
polynomial $P_Y$ attached to $Y$ to be  
\[
\sharp R_{\frg}^{+} - \Dim(Y)=\sharp R_{\frg}^{+}-\half\dim {\mathcal
  O}_{K}=\half(\dim \mathcal{N}-\dim\mathcal{O}_K)=\sharp
R_{\frk}^{+},
\]
%DV {\clrr{
where $\mathcal{N}$ denotes the cone of nilpotent elements in
$\frg^{\star}$. % DV }}  
In other words, the Goldie rank polynomial $P_Y$ has the same degree
as the index polynomial $Q_Y$. 

% Finally, assume that $Y=X\in{\mathcal M}(\lambda_{0})$ and consider
% the coherent family
% $\{X\}\defn\{X_{\lambda}\mid\lambda\in\lambda_{0}+\Lambda\}$ as in
% Section \ref{section coherent}. From the definitions,  
% it follows that ${\mathcal V}(X)={\mathcal V}(X_{\lambda})$ for all
% $\lambda\in\lambda_{0}+\Lambda$. It is known that, for each fixed $j$,
% the map  
%
% \beu
% m_{\{X\}}^j:\lambda\mapsto m^j_{X_{\lambda}} 
% \end{equation*}
%
% extends to a polynomial of degree $\sharp R_{\frk}^{+}$ on
% $\ft^{\star}$ which is proportional  
% to the Goldie rank polynomial $P_{X}$ attached to $X$
% \textcolor{red}{(see, for instance, \cite{BV1}, \cite{Ch}, \cite{J2},
%  \cite{SV})}:  
%
% \beu
% m^j_{\{X\}}=a_{j,X} P_{X}
% \end{equation*}
% %
% where $a_{j,X}$ is a rational number depending on both $X$ and the
% component $\overline{{\mathcal O}_{K}^{j}}$ of ${\mathcal V}(X)$.  
% In particular, we can write
%
% \beu
% P_{X}=\sum_{j}b_{j,X}m^j_{\{X\}},
% \end{equation*}
% with $b_{j,X}$ depending on $j$ and $X$. 
We conjecture that for representations attached to ${\mathcal O}_K$,
the index polynomial admits an expression analogous to \eqref{eq:multchar}. % ,
% but with coefficients  
% independent of $X$. 
\end{subequations} %Korbit
\begin{conj}
\label{conj}
Assume that the $W_{\frg}$-representation $\sigma_K$ generated by the
Weyl dimension formula for $\frk$  
corresponds to a nilpotent $G_{\bC}$-orbit ${\mathcal O}_{K}$ via the
Springer correspondence. 
Then for each $K_\bC$-orbit ${\mathcal O}_{K}^{j}$ on
${\mathcal O}_{K}\cap (\frg/\frk)^{\star}$, there exists an
integer $c_{j}$ such that for any  
Harish-Chandra module $Y$ for $G$ satisfying 
${\mathcal V}(\gr(\Ann(Y))) \subset \overline{O_{K}}$, we have 
\beu
Q_{Y}=\sum_{j}c_{j}m_{Y}^j.
\end{equation*}
Here $Q_{Y}$ is the index polynomial attached to $Y$ as in Section
\ref{section Weyl group}. 
\end{conj}
\begin{ex} 
{\rm
Consider $G=SL(2,\bR)$ with $K=SO(2)$. Then $\sigma_{K}$ is the
trivial representation of $W_{\fg}\simeq\bZ/2\bZ$ and ${\mathcal
  O}_{K}$ is the principal nilpotent orbit.  
${\mathcal O}_{K}$ has two real forms ${\mathcal O}_{K}^{1}$ and
${\mathcal O}_{K}^{2}$. One checks from our computations in Example
\ref{exds_sl2} and from the table below that  
$c_{1}=1$ and $c_{2}=-1$. This shows that the conjecture is true in
the case when $G=SL(2,\bR)$. 
\vspace*{0.2cm}\\
\begin{center}
\begin{tabular}{|l|c|r|}
  \hline
$\hspace*{2.5cm}Y$ & ${\mathcal V}(Y)$ & $Q_{Y}$   \\
  \hline
  finite-dimensional modules & $\{0\}$ & $0$ \\
  \hline
 holomorphic discrete series & ${\mathcal O}_{K}^{1}$ & $1$  \\
  \hline
   antiholomorphic discrete series & ${\mathcal O}_{K}^{2}$ & $-1$  \\
   \hline
   principal series & ${\mathcal O}_{K}^{1}\cup {\mathcal O}_{K}^{2}$ & $0$  \\
   \hline
\end{tabular}
\end{center}
\vspace*{0.5cm}
Here ${\mathcal V}(Y)\subset {\mathcal V}(\gr(\Ann(Y)))$ is the associated variety of $Y$.
}
\end{ex}

\begin{ex}
\label{ex_su1n}
{\rm 
Let $n>1$ and let $G=SU(1,n)$ with $K=U(n)$. Then ${\mathcal O}_K$ is the minimal nilpotent orbit
of dimension $2n$. It has two real forms ${\mathcal O}_{K}^{1}$ and
${\mathcal O}_{K}^{2}$. The holomorphic and 
antiholomorphic discrete series representations $Y^1_\lambda$ and $Y^2_\lambda$ all have Gelfand-Kirillov dimension
equal to $n$. By \cite{Ch}, Corollary 2.13, the respective associated cycles are equal to
\[
\Ass(Y^i_\lambda)=m^i_{Y^i}(\lambda) {\mathcal O}_{K}^{i},\qquad i=1,2,
\]
with the multiplicity $m^i_{Y^i}(\lambda)$ equal to the dimension of the lowest $K$-type of $Y^i_\lambda$. 
The index of the holomorphic discrete series representations is the lowest $K$-type shifted by a one dimensional
representation of $K$ with weight $\rho(\frp^-)$, so it has the same dimension as the lowest $K$-type. The situation for
the antiholomorphic discrete series representations is analogous, but there is a minus sign. Hence
\[
m^i_{Y^i}(\lambda) = (-1)^{i-1}Q_{Y^i}(\lambda),\qquad i=1,2.
\]
This already forces the coefficients $c_1$ and $c_2$ from Conjecture \ref{conj} to be 1 and -1 respectively.

Since ${\mathcal O}_K$ is the minimal orbit, it follows that for infinite-dimensional $Y$,  
\[
{\mathcal V}(\gr(\Ann(Y))) \subseteq \overline{O_{K}}\quad\Rightarrow\quad {\mathcal V}(\gr(\Ann(Y))) = \overline{O_{K}}.
\] 

\medskip

If ${\mathcal V}(\gr(\Ann(Y))) = \overline{O_{K}}$ and $Y$ is irreducible, then $\caV(Y)$ must be either $\overline{\caO_K^1}$ or 
$\overline{\caO_K^2}$. This follows from minimality of $\caO_K$ and from \cite{V3}, Theorem 1.3. Namely, the codimension of the boundary of $\caO_K^i$ in $\overline{\caO_K^i}$ is $n\geq 2$.

On the other hand, by \cite{KO}, Lemma 3.5, $\caV(Y)=\overline{\caO_K^i}$ implies $Y$ is holomorphic if $i=1$, respectively antiholomorphic if $i=2$.
Let us assume $i=1$; the other case is analogous.

It is possible to write $Y$ as a $\bbZ$-linear combination of generalized Verma modules; see for example \cite{HPZ}, Proposition 3.6. 
So we see that it is enough to check the conjecture assuming $Y$ is a generalized Verma module. In this case, one easily computes that
$I(Y)$ is the lowest $K$-type of $Y$ shifted by the one dimensional $\Kt$-module with weight $\rho(\frp^-)$; see \cite{HPZ}, Lemma 3.2. So the index polynomial is the dimension of the lowest $K$-type. By \cite{NOT}, Proposition 2.1, this is exactly the same as the multiplicity $m^1_Y$ of
$\overline{\caO_K^1}$ in the associated cycle. This proves the conjecture in this case (with $c_1=1$). }

\end{ex}

Whenever $G$ is a simple group with a Hermitian symmetric space, the associated varieties ${\mathcal O}_K^1$ and ${\mathcal O}_K^2$ of holomorphic and antiholomorphic discrete series are real forms of a complex orbit ${\mathcal O}_K$ attached by the Springer correspondence to $\sigma_K$. The argument above proves Conjecture \ref{conj} for holomorphic and antiholomorphic representations. But in general there can be many more real forms of ${\mathcal O}_K$, and the full statement of Conjecture \ref{conj} is not so accessible.

\medskip
% DV {\color{blue} 
We mention that neither of the two assumptions (\ref{assum1})
and (\ref{assum2}) above is automatically 
fulfilled. Below, we list the classical groups for which the
assumption (\ref{assum1})  
is satisfied, i.e the classical groups for which $\sigma_{K}$ is a
Springer representation. % } 

% {\color{blue} 
To check whether $\sigma_K$ is a Springer representation, we proceed
as follows (see \cite{Car}, Chapters 11 and 13): 
\begin{itemize}
\item[(i)] we identify $\sigma_K$ as a Macdonald representation;
\item[(ii)] we compute the symbol of $\sigma_K$;
\item[(iii)] we write down the partition associated with this symbol;
\item[(iv)] we check whether the partition corresponds to a complex
  nilpotent orbit. 
\end{itemize} 
Recall that complex nilpotent orbits in classical Lie algebras are in
one-to-one correspondence with the set of partitions  
$\lbrack d_1,\cdots,d_k\rbrack$ with $d_1\geq d_2\geq\cdots\geq
d_k\geq 1$ such that (see \cite{CM}, Chapter 5): 
\begin{itemize}
%  \item[]
\item[$\bullet$] $d_{1}+d_{2}+\cdots+d_{k}=n$, when
  $\frg\simeq\frs\frl(n,\bbC)$; 
\item[$\bullet$] $d_{1}+d_{2}+\cdots+d_{k}=2n+1$ and the even $d_j$
  occur with even multiplicity, when $\frg\simeq\frs\fro (2n+1,\bbC)$; 
\item[$\bullet$] $d_{1}+d_{2}+\cdots+d_{k}=2n$ and the odd $d_j$ occur
  with even multiplicity, when $\frg\simeq\frs\frp(2n,\bbC)$; 
\item[$\bullet$] $d_{1}+d_{2}+\cdots+d_{k}=2n$ and the even $d_j$
  occur with even multiplicity, when $\frg\simeq\frs\fro (2n,\bbC)$;
  except that the partitions having all the
  $d_j$ even and occurring with even multiplicity are each 
  associated to {\em two} orbits. 
\end{itemize}

For example, when $G=SU(p,q)$, with $q\geq p\geq 1$, the Weyl group
$W_\frg$ is the symmetric group $S_{p+q}$, and $W_\frk$ can be
identified with the subgroup $S_p\times S_q$. The representation
$\sigma_K$ is parametrized, as a Macdonald representation, by the
partition  
$\lbrack 2^p,1^{q-p}\rbrack$ (see \cite{M} or Proposition 11.4.1 in
\cite{Car}). This partition corresponds to a $2pq$-dimensional
nilpotent orbit,  
so $\sigma_K$ is Springer. Note that when $\frg$ is of type $A_n$,
there is no symbol to compute, and any irreducible representation of
$W_\frg$  is a Springer representation. 

When $G=SO_e(2p,2p+1)$, with $p\geq 1$, the group $W_\frk$ is
generated by a root subsystem of type $D_p\times B_p$. In this case,  
$\sigma_K$ is parametrized by the pair of partitions $(\lbrack
\alpha\rbrack,\lbrack\beta\rbrack)=(\lbrack
1^p\rbrack,\lbrack1^p\rbrack)$ and its symbol is the array  
\[ 
\begin{pmatrix}
0&&2&&3&&\cdots&&p+1\cr
&1&&2&&\cdots &&p&
\end{pmatrix}.
\]
(See \cite{L} or Proposition 11.4.2 in \cite{Car}.)
The partition of $4p+1$ associated with this symbol is $\lbrack
3,2^{2p-2},1^2\rbrack$. This partition corresponds to a
$2p(2p+1)$-dimensional nilpotent orbit, i.e., $\sigma_K$ is a Springer
representation.  

When $G=Sp(p,q;\bbR)$, with $q> p\geq 1$, the Weyl group $W_\frk$ is
generated by a root subsystem of type $C_p\times C_q$ so that
$\sigma_K$ is parametrized by the pair  
 of partitions $(\lbrack \alpha\rbrack,\lbrack\beta\rbrack)=(\lbrack
 \emptyset\rbrack,\lbrack2^p,1^{q-p}\rbrack)$. Its symbol is the array  
\[
\begin{pmatrix}
0&&1&&2&&\cdots&&q\cr
&1&&2&&\cdots &&q+1&
\end{pmatrix},
\]
where in the second line there is a jump from $q-p$ to $q-p+2$. (See
\cite{L} or Proposition 11.4.3 in \cite{Car}.) 
The partition of $2p+2q$ associated with this symbol is $\lbrack
3,2^{2p-2},1^{2(q-p)+1}\rbrack$. This partition  
does not correspond to a nilpotent orbit, i.e., $\sigma_K$ is not a
Springer representation.

 \scriptsize{\begin{table}[ht]
 \addtolength{\tabcolsep}{-6pt}
%\caption{On Springer assumption} 
\centering 
\scalebox{0.81}{
\begin{tabular}{c c c c c} 
\hline\hline 
& & & &\\${\bf G}$ & {\bf Generator for $\sigma_{K}$}& {\bf Springer ?} & ${\bf {\mathcal O}_{K}}$ &  ${\bf \dim_{\bb C}({\mathcal O}_{K})}$\\[0.5ex]
& & & &\\
% inserts table heading 
\hline\hline
& & & &\\
$SU(p,q)$, $q\geq p\geq 1$&\tiny{$\prod\limits_{\stackrel{1\leq i<j\leq
      p}{p+1\leq i<j\leq p+q}}(X_{i}-X_{j})$ for $p\geq 2$}
&\tiny{Yes}&\tiny{$\lbrack 2^p,1^{q-p}\rbrack$}&\tiny{$2pq$} \\[8ex] 
& $\prod\limits_{2\leq i<j\leq q+1}(X_{i}-X_{j})$ for $q\geq 2$, $p=1$ &  &
(minimal orbit if $p=1$)&  \\[5ex] 
&  $\sigma_{K}$ is trivial for $p=q=1$&  & (principal orbit if $p=q=1$)&  \\
& & & &\\ 
\hline\hline
& & & &\\
$SO_{e}(2p,2p+1)$, $p\geq 1$& $\prod\limits_{\stackrel{1\leq i<j\leq p}{p+1\leq
    i<j\leq 2p}}(X_{i}^{2}-X_{j}^{2})\prod\limits_{p+1\leq i\leq 2p}X_{i}$ for
$p\geq 2$&Yes &$\lbrack 3,2^{2p-2},1^{2}\rbrack$&  
$2p(2p+1)$ \\[8ex] 
& $X_{2}$ for $p=1$& &(subregular orbit if $p=1$) &  \\ 
& & & &\\
\hline\hline
& & & &\\
$SO_{e}(2p,2p-1)$, $p\geq 1$\;\;\;\;\;\;&
$\prod\limits_{\stackrel{1\leq   i<j\leq p}{p+1\leq i<j\leq 
    2p-1}}(X_{i}^{2}-X_{j}^{2})\prod\limits_{i=p+1}^{2p-1}% {p+1\leq 
%  i\leq 2p-1} 
X_{i}$ for $p\geq 2$&Yes &$\lbrack 3,2^{2p-2}\rbrack$&  
$2p(2p-1)$ \\  
& $\sigma_{K}$ is trivial for $p=1$& &(principal orbit if $p=1$)&  \\ 
& & & &\\
\hline\hline
& & & &\\
$SO_{e}(2,2q+1)$, $q\geq 2$& $\prod\limits_{2\leq i<j\leq
  q+1}(X_{i}^{2}-X_{j}^{2})\prod\limits_{i=2}^{q+1} % 2\leq i\leq q+1}
X_{i}$&Yes &$\lbrack
3,1^{2q}\rbrack$&  
$2(2q+1)$ \\ 
& & & &  \\ 
%& $\sigma_{K}$ is trivial for $q=0$& &(principal orbit if $q=0$)&  \\ 
& & & &\\
\hline\hline
& & & &\\
$SO_{e}(2p,2q+1)$ &$\prod\limits_{\stackrel{1\leq i<j\leq p}{p+1\leq i<j\leq
    p+q}}(X_{i}^{2}-X_{j}^{2})\prod\limits_{i=p+1}^{p+q}X_{i}$ &Yes &
$\lbrack 3,2^{2p-2},1^{2(q-p)+2}\rbrack$ &$2p(2q+1)$ \\  
$q\geq p+1\geq 3$& & &  &  \\ 
& & & &\\
\hline\hline
& & & &\\
$SO_{e}(2p,2q+1)$ &$\prod\limits_{\stackrel{1\leq i<j\leq p}{p+1\leq i<j\leq
    p+q}}(X_{i}^{2}-X_{j}^{2})\prod\limits_{i=p+1}^{p+q} % p+1\leq
                                % i\leq p+q} 
X_{i}$ for $q\geq 2$&No & {\huge -} &{\huge -} \\[8ex]  
$p\geq q+2\geq 2$ & $X_{p+1}\prod\limits_{1\leq i<j\leq
  p}(X_{i}^{2}-X_{j}^{2})$ for $q=1$& & &  \\[5ex] 
& $\prod\limits_{1\leq i<j\leq p}(X_{i}^{2}-X_{j}^{2})$ for $q=0$& & &\\ 
& & & &\\
\hline\hline
& & & &\\
$Sp(2n,\bb R)$, $n\geq 1$ &$\prod\limits_{1\leq i<j\leq n}(X_{i}-X_{j})$ for
$n\geq 2$& Yes & $\lbrack 2^n\rbrack$&$n(n+1)$  \\[5ex]  
&  $\sigma_{K}$ is trivial for $n=1$&  & (principal orbit if $n=1$)&  \\ 
& & & &\\
\hline\hline
& & & &\\
$Sp(p,q;\bb R)$, $q\geq p\geq 1$ & \tiny{$\prod\limits_{\stackrel{1\leq i<j\leq
      p}{p+1\leq i<j\leq
      p+q}}(X_{i}^{2}-X_{j}^{2})\prod\limits_{i=1}^{p+q} % 1\leq i\leq
                                % p+q}
X_{i}$ for $p\geq 2$}  & No & {\huge -} &  {\huge -} \\[8ex]  
& $\prod\limits_{2\leq i<j\leq q+1}(X_{i}^{2}-X_{j}^{2})\prod\limits_{1\leq i\leq
  q+1}X_{i}$ for $q\geq 2$, $p=1$&   &  \\[5ex]  
&  $X_{1}X_{2}$ for $p=q=1$ &  &&  \\ 
& & & &\\ 
\hline\hline 
& & & &\\
$SO_{e}(2p,2q)$, $q\geq p\geq 1$ & $\prod\limits_{\stackrel{1\leq i<j\leq
    p}{p+1\leq i<j\leq p+q}}(X_{i}^{2}-X_{j}^{2})$ for $p\geq 2$& Yes
&$\lbrack 3,2^{2p-2},1^{2(q-p)+1}\rbrack$&$4pq$ \\[8ex]  
& $\prod\limits_{2\leq i<j\leq q+1}(X_{i}^{2}-X_{j}^{2})$ for $q\geq
2$, $p=1$ &  & &  \\[5ex] 
& $\sigma_{K}$ is trivial for $p=q=1$& & (principal orbit for $p=q=1$) &   \\ 
& & & &\\ 
\hline\hline
& & & &\\
$SO^\star(2n)$, $n\geq 1$ & $\prod\limits_{1\leq i<j\leq n}(X_{i}-X_{j})$ for
$n\geq 2$& Yes & $\lbrack 2^{n}\rbrack$ for $n$ even & $n(n-1)$ \\[5ex]
&   $\sigma_{K}$ is trivial for $n=1$ &  & $\lbrack
2^{n-1},1^{2}\rbrack$ for $n$ odd &  \\  
&  & &   (trivial orbit if $n=1$)&  \\ 
&  &  & (minimal orbit if $n=3$)&  \\ [1ex]
\hline\hline\end{tabular} }\label{tableSpringer}
\end{table}}

\normalsize  
% DV 
\clearpage
The following theorem provides a sufficient condition for
  both assumptions (\ref{assum1}) and (\ref{assum2}) to hold.  
In contrast with the previous table, it includes exceptional groups.
\begin{thm}\label{OK}
Suppose G is connected semisimple, $T$ is a compact Cartan subgroup in
$G$ contained in $K$, and $\lambda_0$ is the Harish-Chandra  
parameter for a discrete series representation $Y_0$ of
$G$. Assume that the set of integral roots for $\lambda_0$ is precisely  
the set of compact roots, i.e., 
\begin{equation}\label{exception}
\{\alpha \in \Delta(\frg,\frt)|\lambda_0(\alpha^\vee) \in {\mathbb Z} \}=\Delta(\frk,\frt).
\end{equation}
Then $\sigma_K$ is the Springer representation for a complex nilpotent
orbit ${\mathcal O}_K$.  
Let $\{Y_{\lambda_0+\mu} | \mu\in\Lambda\}$ be the Hecht-Schmid
coherent family of virtual representations corresponding  
to $Y_0$ and form the virtual representation
$$Y \defn \sum_{w \in W_\frk} (-1)^w Y_{w\lambda}.$$
Then $Y$ is a nonzero integer combination of irreducible representations
having associated variety of annihilator equal to $\overline{{\mathcal
O}_K}$. 
\end{thm} 
\begin{proof}
The character of $Y$ on the compact Cartan $T$ is a multiple (by the
cardinality of $W_\frk$) of the character of $Y_0$. Consequently the
character of $Y$ on $T$ is not zero, so $Y$ is not zero.  By
construction the virtual representation $Y$ transforms under the
coherent continuation action of the integral Weyl group $W(\lambda_0)
= W_\frk$ by the sign character of $W(\lambda_0)$.  By the theory of
$\tau$-invariants of Harish-Chandra modules, it follows that every
irreducible constituent of $Y$ must have every simple integral root in
its $\tau$-invariant.

At any regular infinitesimal character $\lambda_0$ there is a unique
maximal primitive ideal $J(\lambda_0)$, characterized by having every
simple integral root in its $\tau$ invariant. The Goldie rank
polynomial for this ideal is a multiple of 
\[
q_0(\lambda) = \prod_{\langle\alpha^\vee,\lambda_0 \rangle \in \bbN}
\langle \alpha^\vee,\lambda\rangle;
\]
so the Goldie rank polynomial for every irreducible constituent of $Y$
is a multiple of $q_0$.  The Weyl group representation generated by
$q_0$ is $\sigma_K$ (see \eqref{Korbit}); so by \cite{BV1}, it follows
that the complex nilpotent orbit ${\mathcal O}_0$ attached to the
maximal primitive ideal $J_0$ must correspond to $\sigma_K$ as in
\eqref{assum1}. At the same time, we have seen that the (nonempty!)
set of irreducible constituents of the virtual representation $Y$ all
satisfy \eqref{assum2}.

% BY primitive ideal theory, one gets an integer combination of
% irreducible representations satisfying (\ref{assum2}).   ???
\end{proof} 

Theorem \ref{OK} applies to any real form of $E_6$, $E_7$ and $E_8$,
and more generally to any equal rank real form of one root length. It
applies as well to $G_2$ (both split and compact forms. The theorem
applies to compact forms for any $G$, and in that case ${\mathcal O}_K
={0}$).  However, for the split $F_4$ and taking $\lambda_0$ a
discrete series parameter for the nonlinear double cover, the integral
root system (type $C_4$) strictly contains the compact roots (type
$C_3 \times C_1$).  So the above theorem does not apply to split
$F_4$. Nevertheless the representation $\sigma_K$ does correspond to a
(special) nilpotent orbit ${\mathcal O}_K$. At regular integral
infinitesimal character, there are (according to the
representation-theoretic software {\tt atlas}; see \cite{atlas})
exactly $27$ choices for an irreducible representation $Y$ as in
(\ref{assum2}).  There are two real forms of the orbit ${\mathcal
  O}_K$. The $Y$'s come in three families (``two-sided cells") of nine
representations each, with essentially the same associated variety in
each family. One of the three families contains an $A_\frq(\lambda)$
(with Levi of type $B_3$) and therefore has associated variety equal
to one of the two real forms. In particular, the condition
(\ref{exception}) is sufficient but not necessary for assumptions
(\ref{assum1}) and (\ref{assum2}) to hold. Note that for rank one
$F_4$, the representation $\sigma_K$ is not in the image of the
Springer correspondence.

For the classical groups, Theorem \ref{OK} applies to all the cases of
one root length, explaining all the ``yes'' answers in Table
\ref{tableSpringer} for types $A$ and $D$. In the case of two root
lengths, the hypothesis of Theorem \ref{OK} can be satisfied in the
noncompact case exactly when $G$ is Hermitian symmetric (so the cases
$SO_e(2,2n-1)$ and $Sp(2n,\bbR)$; more precisely, for appropriate
nonlinear coverings of these groups).  

We do not know a simple general explanation for the remaining ``yes''
answers in the table. Just as for $F_4$, the integral root systems for
a discrete series parameter $\lambda_0$ are too large for Theorem
\ref{OK}: in the case of $SO_e(2p,2q+1)$, for example, the root system
for $K$ is $D_p\times B_q$, but (for $p\ge 2$) the integral root system cannot be
made smaller than $B_p \times B_q$.

\end{document}